\documentclass[11pt, reqno]{amsart}
\usepackage{amssymb, amsthm, amsmath, amsfonts,mathtools}
\usepackage{array, epsfig}
\usepackage{bbm}
\usepackage{hyperref}
\usepackage{url}
\usepackage[numbers,square]{natbib}
\usepackage{color}
\allowdisplaybreaks

\numberwithin{equation}{section}
  \evensidemargin
\oddsidemargin
\DeclareMathOperator{\Row}{Row}

\newcommand{\rank}{\mathop{\mathrm{rank}}\nolimits}

\newcommand{\bfA}{\mathbf{A}}
\newcommand{\bfB}{\mathbf{B}}
\newcommand{\bfG}{\mathbf{G}}
\newcommand{\bfH}{\mathbf{H}}

\newcommand{\bS}{\mathbb{S}}

\newcommand{\cW}{\mathcal{W}}

\newcommand{\cP}{\mathcal{P}}

\newcommand{\cS}{\mathcal{S}}

\DeclareMathOperator*{\Ker}{Ker}

\def\dint{\textup{d}}

\newcommand{\E}{\mathbb E}

\newcommand{\R}{\mathbb{R}}
\newcommand{\N}{\mathbb{N}}

\newcommand{\C}{\mathbb{C}}

\renewcommand{\P}{\mathbb{P}}

\renewcommand{\Im}{\operatorname{Im}}

\newcommand{\Int}{\mathop{\mathrm{Int}}\nolimits}
\newcommand{\Var}{\mathop{\mathrm{Var}}\nolimits}

\newcommand{\pos}{\mathop{\mathrm{pos}}\nolimits}
\newcommand{\lin}{\mathop{\mathrm{lin}}\nolimits}

\newcommand{\eps}{\varepsilon}
\newcommand{\eqdistr}{\stackrel{d}{=}}

\newcommand{\todistr}{\overset{w}{\underset{n\to\infty}\longrightarrow}}

\newcommand{\bsl}{\backslash}
\newcommand{\ind}{\mathbbm{1}}

\newcommand{\dd}{{\rm d}}
\newcommand{\eee}{{\rm e}}

\theoremstyle{plain}
\newtheorem{theorem}{Theorem}[section]
\newtheorem{lemma}[theorem]{Lemma}
\newtheorem{corollary}[theorem]{Corollary}
\newtheorem{proposition}[theorem]{Proposition}
\newtheorem{conjecture}[theorem]{Conjecture}

\theoremstyle{definition}
\newtheorem{definition}[theorem]{Definition}
\newtheorem{example}[theorem]{Example}

\theoremstyle{remark}
\newtheorem{remark}[theorem]{Remark}

\makeatletter
\@namedef{subjclassname@2020}{\textup{2020} Mathematics Subject Classification}
\makeatother

\begin{document}

\author{Zakhar Kabluchko}
\address{Zakhar Kabluchko: Institut f\"ur Mathematische Stochastik,
Universit\"at M\"unster,
Orl\'eans-Ring~10,
48149 M\"unster, Germany}
\email{zakhar.kabluchko@uni-muenster.de}

\title[Random Polyhedral Cones]{Random Polyhedral Cones I: \\ Distributional Results via Gale Duality}


  %I checked this

\begin{abstract}
Let $U_1,\ldots,U_n$ be independent random vectors uniformly distributed on the unit sphere
$\mathbb S^{d-1}\subseteq\mathbb R^d$, where $n\ge d$, and consider the random polyhedral cone
\[
\mathcal W_{n,d}:=\mathop{\mathrm{pos}} (U_1,\ldots,U_n) = \{\lambda_1 U_1+ \ldots + \lambda_n U_n: \lambda_1\geq 0, \ldots, \lambda_n \geq 0\}.
\]
We establish several distributional results for $\mathcal W_{n,d}$ and the associated spherical
polytope $\mathcal W_{n,d}\cap\mathbb S^{d-1}$. Our main contributions include:
\begin{itemize}
\item[(i)] Let $\alpha_d$ denote the solid angle of $\mathcal W_{d,d}$ and write $m(d,k):=\E[\alpha_d^k]$
for its $k$-th moment. We prove the symmetry $m(d,k)=m(k,d)$. As an application, we compute
$\Var[\alpha_d]=2^{-d}(d+1)^{-1}-4^{-d}$ and derive a closed formula for the third moment.
\item[(ii)] For $n=d+1,d+2,d+3$ we determine the probability that $\mathcal W_{n,d}\cap\mathbb S^{d-1}$ is a spherical simplex,
a spherical analogue of the classical Sylvester problem. In the case $n=d+2$ we also determine the distribution of the number of
vertices of $\mathcal W_{d+2,d}\cap\mathbb S^{d-1}$.
\item[(iii)] Let $f_\ell(\mathcal W_{n,d})$ denote the number of $\ell$-dimensional faces of $\mathcal W_{n,d}$.
We prove a distributional limit theorem for $f_\ell(\mathcal W_{n,d})$ in the regime $n=d+k$ and $\ell=d-q$,
where $k,q\in\mathbb N$ are fixed and $d\to\infty$. The limit law is a centered weighted sum of independent
chi-squared variables, with weights determined by explicit
eigenvalues of a convolution operator on the sphere. 
\item[(iv)] Let $I_1,I_2 \subseteq [n]$ be such that $I_1\cup I_2 = [n]$  and $\max(\#I_1, \#I_2) \leq d-1$. We prove that  the ``face events'' $\{\,\mathop{\mathrm{pos}} (U_i: i\in I_r) \text{ is a face of } \mathcal W_{n,d}\,\}$, \  $r=1,2$,  are independent.
\end{itemize}

A unifying ingredient is an explicit coupling that,
for $n>d$, produces i.i.d.\ uniform vectors
$U_1,\ldots,U_n\in\mathbb S^{d-1}$ together with
i.i.d.\ uniform vectors
$V_1,\ldots,V_n\in\mathbb S^{n-d-1}$
whose associated oriented matroids are Gale dual;
proper faces of $\pos(U_1,\ldots,U_n)$ correspond
to complementary subconfigurations of
$V_1,\ldots,V_n$ that positively span
$\mathbb R^{n-d}$. This correspondence yields, in particular, a
$U$-statistic representation of
$f_\ell(\mathcal W_{n,d})$.
\end{abstract}

\maketitle
\tableofcontents

\section{Introduction}

\subsection{Random polyhedral cones}
Let \(U_1,\ldots, U_n\) be independent random vectors, each distributed uniformly on the unit sphere $\bS^{d-1}$ in \(\R^d\).
The main object of interest in the present paper is the random polyhedral cone given by the positive hull of these vectors,
\[
\cW_{n,d}:= \pos (U_1,\ldots, U_n)
:= \{\lambda_1 U_1+ \cdots + \lambda_n U_n : \lambda_1\geq 0,\ldots,\lambda_n\geq 0\}\subseteq \R^d .
\]
The polar (convex dual) cone of \(\cW_{n,d}\) is the intersection of the half-spaces with outer unit normal vectors
\(U_1,\ldots, U_n\), namely
\[
\cW_{n,d}^\circ
:= \{y\in \R^d : \langle x, y \rangle\leq 0 \text{ for all } x\in \cW_{n,d}\}
= \bigcap_{i=1}^n  \{y\in \R^d : \langle y, U_i \rangle \leq 0\}.
\]
Thus, \(\cW_{n,d}^\circ\) is the feasible set of a system of \(n\) random linear inequalities, making it a natural and fundamental object.
The random cones \(\cW_{n,d}\), \(\cW_{n,d}^\circ\), and related models have been studied by
\citet{CoverEfron}, \citet{donoho_tanner_orthants}, and \citet{hug_schneider_conical_tessellations};
see also
\cite{godland_kabluchko_thaele_cones_high_dim_part_I, HugSchneiderThresholdPhenomenaPart2,HugSchneiderThresholdPhenomena,kabluchko_thaele_faces_hypersphere_tess,Schneider2021RandomGale}
and Chapters~5--6 of the book by \citet{schneider_book_convex_cones_probab_geom}.
One important result is Wendel's formula~\cite{wendel} (see also~\cite[Theorem~8.2.1]{schneider_weil_book}), which states that
\begin{equation}\label{eq:wendel_formula}
p(n,d) := \P[\cW_{n,d} = \R^d]
=  \frac 1  {2^{n-1}} \sum_{\ell=d}^{n-1} \binom{n-1}{\ell}.
\end{equation}

Beyond this, explicit formulas are known for expectations of several geometric functionals of \(\cW_{n,d}\) and \(\cW_{n,d}^\circ\)
(as well as of their conditioned versions).
As one example, for \(\ell \in \{0,\ldots, d-1\}\), let \(f_{\ell} (\cW_{n,d})\) denote the number of \(\ell\)-dimensional faces of \(\cW_{n,d}\).
Then, for $n>d$, \citet[Theorem~1.6]{donoho_tanner_orthants}
show that
\[
\E f_\ell (\cW_{n,d})
= \binom{n}{\ell}\, p(n-\ell, n-d)
= \frac{1}{2^{\,n-\ell-1}}\binom{n}{\ell}\sum_{p=n-d}^{n-\ell-1}\binom{n-\ell-1}{p}.
\]
Other geometric functionals for which explicit expectation formulas are available include the conic intrinsic volumes
(with the solid angle \(\alpha(\cW_{n,d})\) as a special case), as well as sums of conic intrinsic volumes over all faces of a given dimension
and sums over the tangent cones at these faces.
There also exist sporadic results on second moments of certain geometric functionals~\cite[Theorem~8.1]{hug_schneider_conical_tessellations}, but almost nothing seems to be known about the full
distribution of \(f_\ell(\cW_{n,d})\), \(\alpha(\cW_{n,d})\), and related quantities.

\subsection{Main results}
In the present paper we establish several distributional results for \(\cW_{n,d}\) that go beyond expectations and second moments.

\subsubsection*{Moments of the solid angle}
In Section~\ref{sec:moments_angles} we study the solid angle of the cone \(\cW_{d,d}\) (i.e., \(n=d\)).
For a full-dimensional polyhedral cone \(C\subseteq \R^d\), its solid angle is defined by
\(\alpha(C):=\P[U\in C]\), where \(U\) is uniformly distributed on the unit sphere in \(\R^d\).
Writing \(m(d,k):=\E[\alpha(\cW_{d,d})^k]\) for the \(k\)-th moment, we prove the symmetry relation
\(m(d,k)=m(k,d)\).
As an application, we compute
\[
\Var(\alpha(\cW_{d,d}))=2^{-d}(d+1)^{-1}-4^{-d}.
\]
We also derive a closed-form expression for the third moment and propose a conjecture for the asymptotic behavior of \(m(d,k)\) as \(d\to\infty\) (with fixed $k$).

\subsubsection*{Spherical Sylvester problem}
In Section~\ref{sec:spherical_sylvester} we address a spherical analogue of Sylvester-type questions.
For \(n=d+1,d+2,d+3\) we compute the probability that \(\cW_{n,d}\cap \mathbb S^{d-1}\) is a spherical simplex.
Moreover, for \(n=d+2\) we determine the distribution of the number of vertices of
\(\cW_{d+2,d}\cap \mathbb S^{d-1}\).

\subsubsection*{High-dimensional limit theorems for face counts}
In Section~\ref{sec:limit_theorems_random_cones_high_dim} we prove a distributional limit theorem for face counts in a high-dimensional regime.
Recall that  \(f_{\ell}(\cW_{n,d})\) denotes the number of \(\ell\)-dimensional faces of \(\cW_{n,d}\).
We consider \(n=d+k\) and \(\ell=d-q\), where \(k\in\N\) and \(q\in\N\) are fixed, and let \(d\to\infty\).
We show that the centered and normalized quantity \(f_{d-q}(\cW_{d+k,d})\) converges in distribution to a non-Gaussian limit; more precisely,
\[
d\cdot \left(\frac{f_{d-q}(\cW_{d+k,d})}{\binom {d+k}{d-q}} - p(k+q, k)\right)
\;{\overset{w}{\underset{d\to\infty}\longrightarrow}}\;
\frac {\binom{k+q}{2}\binom {k+q-2}{k-1}} {2^{k+q-1}}  \, (1 -Q_k),
\]
where \(Q_k>0\) is an infinitely divisible random variable with \(\E Q_k = 1\) that admits the representation
\[
Q_k \eqdistr \sum_{r=1,3,5,\ldots}
\left(\frac{\Gamma\!\left(\frac{k}{2}\right)\Gamma\!\left(\frac{r}{2}\right)}
{\pi\,\Gamma\!\left(\frac{r+k}{2}\right)}\right)^2
\operatorname{Gamma} \left(\frac {d_{r,k}}2, \frac 12\right),
\qquad k\ge 2,
\]
and \(Q_1 \eqdistr \operatorname{Gamma}(\tfrac 12, \tfrac 12)\).
Here \(\operatorname{Gamma}(\alpha,\lambda)\) denotes a Gamma random variable with shape parameter \(\alpha>0\) and rate parameter \(\lambda>0\),
and all such variables are independent.
Moreover,
\[
d_{r,k} := \binom{r+k-1}{r}-\binom{r+k-3}{r-2}
\]
is the dimension of the space of degree-\(r\) spherical harmonics on \(\mathbb S^{k-1}\subseteq \R^k\) for \(k\ge 2\).
A more detailed statement is given in Theorem~\ref{theo:distributional_limit_face_count_random_cone}.

\subsubsection*{Independence of face events}
For a subset \(I\subseteq \{1,\ldots,n\}\), define the cone \(F_I := \pos(U_i: i\in I)\), which is a candidate face of \(\cW_{n,d}\),
and the event
\[
A_I := \{F_I \text{ is a face of } \cW_{n,d}\}.
\]
Somewhat surprisingly, many of the ``face events'' \(A_I\) are independent.
In Section~\ref{subsec:face_events_independence} we provide a sufficient condition guaranteeing such independence.

\subsection{Common technique: Gale duality}
Although the results above may appear unrelated at first sight, their proofs share a common ingredient: passing to the (linear) Gale transform.
The necessary background on Gale transforms is collected in Section~\ref{subsec:gale_transforms}.

The present work was inspired by \citet{frick_newman_pegden}. Their key observation is that the affine Gale transform of
\(n\) i.i.d.\ standard Gaussian points in \(\R^d\) may be realized as \(n\) independent standard Gaussian points in \(\R^{\,n-d-1}\),
translated so that their barycenter is at the origin. Using this, they established an equivalence between the following two problems:
\begin{itemize}
\item[(a)] determining the distribution of the Radon type of \(d+2\) independent Gaussian points in \(\R^d\), and
\item[(b)] determining the location of the sample mean among the order statistics of a one-dimensional i.i.d.\ Gaussian sample of size \(d+2\).
\end{itemize}
For further results on these two problems, see
\cite{Barysheva2025RandomConvexHulls,chan_kalai_etal_unimodality_radon_partitions,kabluchko_panzo,kuchelmeister_youdens_demon,white_radon_partitions_gauss_polys}.

In contrast, the Gale duality input used in the present paper is linear rather than affine.
The corresponding results will be stated in Proposition~\ref{prop:gauss_proj_gauss_embed_coupled} and
Theorem~\ref{theo:gale_duality_stoch_orthant}.

\subsection{Notation}
Throughout the paper we use the following notation.
For a finite set \(A\), let \(\#A\) denote its cardinality. For \(n\in\N\), write \([n]\coloneqq \{1,\ldots,n\}\).

Let \(\R^d\) be \(d\)-dimensional Euclidean space with the standard inner product \(\langle\cdot,\cdot\rangle\) and norm \(\|\cdot\|\).
Let \(\bS^{d-1}:=\{x\in\R^d:\|x\|=1\}\) be the unit sphere, and let \(e_1,\ldots,e_d\) denote the standard orthonormal basis of \(\R^d\).

For \(X\subseteq \R^d\), we write \(\pos(X)\) for the positive hull of \(X\) and \(\lin(X)\) for the linear span of \(X\), that is,
\begin{align*}
\pos(X)
&:=
\left\{\sum_{i=1}^m \lambda_i x_i \,:\, m\in\mathbb N_0,\ x_1,\ldots,x_m\in X,\ \lambda_1,\ldots,\lambda_m\ge 0\right\},
\\
\lin(X)
&:=
\left\{\sum_{i=1}^m \lambda_i x_i \,:\, m\in\mathbb N_0,\ x_1,\ldots,x_m\in X,\ \lambda_1,\ldots,\lambda_m\in\R\right\}.
\end{align*}
If \(X=\{x_1,\ldots,x_n\}\) is finite, we also write \(\pos(x_1,\ldots,x_n)\) and \(\lin(x_1,\ldots,x_n)\) instead of \(\pos(X)\) and \(\lin(X)\).
By convention, \(\pos(\varnothing)=\lin(\varnothing)=\{0\}\).

A \emph{polyhedral cone} is the positive hull of finitely many vectors in \(\R^d\).
Equivalently, it is an intersection of finitely many half-spaces of the form
\(\{x\in\R^d:\langle x,u\rangle\le 0\}\), where \(u\in\R^d\setminus\{0\}\).

For sequences \((a_n)_{n\in\N}\) and \((b_n)_{n\in\N}\), we write \(a_n\sim b_n\) if \(a_n/b_n\to 1\) as \(n\to\infty\).
For random variables (or random elements) \(X\) and \(Y\), we write \(X\eqdistr Y\) to denote equality in distribution.
Weak convergence of random variables or elements is denoted by \(\overset{w}{\longrightarrow}\).

\section{Linear Gale duality for random vector configurations}
\subsection{Linear Gale duality}\label{subsec:gale_transforms}
Gale duality (or the Gale transform) is a widely used tool in convex and discrete geometry. For detailed accounts of Gale duality and its applications, we refer to the books of~\citet[Section~5.4]{gruenbaum_book}, \citet[Chapter~6]{ziegler_book_lectures_polytopes}, ~\citet[Section~5.6]{matousek_book_lect_discrete_geometry}, \citet[Chapter~3]{mcmullen_shephard_book}, \citet[Section~3.4]{bjoerner_etal_book_oriented_matroids}, \citet[Section~2.14]{PinedaVillavicencio_book_2024PolytopesGraphs}, \citet[Sections~6B, 6C]{McMullen_book_2026ConvexPolytopesPolyhedra} and to the papers of~\citet{mcmullen_transforms}, \citet{Shephard1971DiagramsPositiveBases},  and~\citet{grunbaum_shephard_convex_polytopes_1969}.
There are two versions of Gale duality: affine and linear. The version most commonly treated in the literature is the affine one; in the present work, however, we require linear Gale duality studied in~\cite{mcmullen_transforms,Shephard1971DiagramsPositiveBases}, which we now define.

\begin{definition}[Linear Gale duality]\label{def:gale_duality_linear}
Let $d',d'' \in \mathbb{N}$ and put $n := d' + d''$.
Two vector configurations
$a_1,\dots,a_n \in \mathbb{R}^{d'}$ and
$b_1,\dots,b_n \in \mathbb{R}^{d''}$
are said to be in \emph{linear Gale duality} if the following conditions hold:
  \begin{enumerate}
    \item[(i)] the linear span of $a_1,\dots,a_n$ is $\R^{d'}$ and the linear span of $b_1,\dots,b_n$  is $\R^{d''}$;
    \item[(ii)] if $A \in \mathbb{R}^{d' \times n}$ is the matrix with columns $a_1,\ldots, a_n$ and $B \in \mathbb{R}^{d'' \times n}$ is the matrix with columns $b_1,\ldots, b_n$, then every row of $A$ is orthogonal to every row of $B$. In matrix notation,
          \[
            A B^{\top} \;=\; 0_{d' \times d''}.
          \]
  \end{enumerate}
\end{definition}

Condition~(i) states that $\rank A = d'$ and $\rank B = d''$. The row spaces of $A$ and $B$, denoted by $\Row(A)$ and $\Row(B)$, are defined as the linear subspaces in $\R^n$ spanned by the rows of $A$ and $B$, respectively. Condition~(ii) states that $\Row(A)$ and $\Row(B)$ are orthogonal to each other. In fact, by the full rank assumption, they even form complementary orthogonal subspaces of $\mathbb{R}^n$.  Thus conditions (i) and (ii) can be stated as
\begin{equation}\label{eq:Gale_duality_row_spaces_orthogonal_intro}
\dim\Row(A) = d', \qquad \dim\Row(B)=d'',\qquad  \Row(A)=\Row(B)^\perp.
\end{equation}

There is a pronounced correspondence between the properties of linearly Gale-dual configurations
$a_1,\dots,a_n$ and $b_1,\dots,b_n$.   The following Lemma~\ref{lem:faces-positive-spanning-intro-version} records a concrete
instance of this correspondence: faces of the cone generated by the $a_i$ correspond to complements of
positive spanning subsets of the $b_j$, and conversely. While this fact is standard and can be found in some or other form in the references cited above, we
include a complete elementary proof in Appendix~\ref{sec:gale_duality_basic_facts},  Lemma~\ref{lem:faces-positive-spanning}.

\begin{lemma}[Faces vs.\ positive spanning subsets under general linear position]\label{lem:faces-positive-spanning-intro-version}
Let $d',d''\in\mathbb N$ and put $n:=d'+d''$. Let $a_1,\dots,a_n\in\mathbb R^{d'}$ and
$b_1,\dots,b_n\in\mathbb R^{d''}$ be in linear Gale duality, and assume that both
configurations are in general linear position. Let $I\subseteq[n]$ satisfy $\# I\le d'-1$.
Then the following are equivalent:
\begin{enumerate}
  \item[(i)] $\pos(a_i:i\in I)$ is a face of the polyhedral cone $\pos(a_1,\dots,a_n)$.
  \item[(ii)] $\pos(b_j:j\in I^c)=\mathbb R^{d''}$. Here $I^c := [n] \bsl I$ is the complement of $I$.
\end{enumerate}
\end{lemma}

\begin{remark}
A finite collection of vectors $v_1,\dots,v_m\in\mathbb R^d$
is said to be in \emph{general linear position}
if, for every $J\subseteq[m]$ with $\#J\le d$,
the vectors $(v_j)_{j\in J}$ are linearly
independent. If $m\ge d$, this is equivalent
to requiring that every $d$-element subfamily
$(v_j)_{j\in J}$, $J\subseteq[m]$, $\#J=d$,
is linearly independent.
Note that for $d\ge 2$ this condition implies
that the vectors are pairwise distinct, whereas
for $d=1$ it merely requires that all vectors
are nonzero.
\end{remark}

\begin{example}\label{ex:I-empty}
Let $I=\varnothing$. By convention, $\pos(\varnothing)=\{0\}$, so condition~(i) says that $\{0\}$ is a face of
$C_A:=\pos(a_1,\ldots,a_n)$, i.e.\ that the cone $C_A$ is pointed. If the vectors
$a_1,\ldots,a_n$ are in general linear position and linearly span $\R^{d'}$, then one checks
that $C_A$ is either pointed or equal to $\R^{d'}$. Hence  Lemma~\ref{lem:faces-positive-spanning-intro-version} with $I=\varnothing$ takes the form
\[
\pos(a_1,\ldots,a_n)\neq \R^{d'} \quad\Longleftrightarrow\quad \pos(b_1,\ldots,b_n)=\R^{d''}.
\]
\end{example}

\subsection{Gale coupling between two Gaussian projections}
Let $\bfG$ be a \emph{Gaussian $d_1\times d_2$ matrix}, i.e.\ a random matrix whose
entries are i.i.d.\ standard Gaussian random variables. We view $\bfG$ as a random linear operator
$
\bfG:\mathbb R^{d_2}\to \mathbb R^{d_1},
$
and refer to it as a \emph{Gaussian linear operator}. Its transpose $\bfG^{\top}:\mathbb R^{d_1}\to
\mathbb R^{d_2}$ is again a Gaussian linear operator.
A key property is rotational invariance: for any orthogonal transformations
$O_1\in O(d_1)$ and $O_2\in O(d_2)$, the random matrix $O_1\,\bfG\,O_2$ has the same
distribution as $\bfG$. In particular, the notion of a Gaussian linear operator does not
depend on the choice of orthonormal bases in $\mathbb R^{d_1}$ and $\mathbb R^{d_2}$.
If $d_2\ge d_1$, we also refer to $\bfG$ as a \emph{Gaussian projection}; note, however, that
$\bfG$ is not a projection in the literal sense.

Fix $d',d''\in\mathbb N$ and set $n:=d'+d''$. Let $\bfG:\mathbb R^{n}\to\mathbb R^{d'}$ and
$\bfH:\mathbb R^{n}\to\mathbb R^{d''}$ be Gaussian projections. With probability one, both
$\Ker(\bfG)$ and $\Ker(\bfH)^\perp$ are $d''$-dimensional linear subspaces of $\mathbb R^{n}$
(since $\bfG$ has rank $d'$ and $\bfH$ has rank $d''$ almost surely). Let $G(n,d'')$ denote the
Grassmannian of $d''$-dimensional linear subspaces of $\mathbb R^{n}$, endowed with its natural
Borel $\sigma$-algebra. We view $\Ker(\bfG)$ and $\Ker(\bfH)^\perp$ as random elements of
$G(n,d'')$.
By rotational invariance of Gaussian matrices, these random subspaces are $O(n)$-invariant: for
every orthogonal transformation $O\in O(n)$,
\[
O\,\Ker(\bfG)\ \eqdistr\ \Ker(\bfG)
\qquad\text{and}\qquad
O\,\Ker(\bfH)^\perp\ \eqdistr\ \Ker(\bfH)^\perp.
\]
Since there is a unique $O(n)$-invariant probability measure on $G(n,d'')$ (the Haar, or
uniform, measure), it follows that both $\Ker(\bfG)$ and $\Ker(\bfH)^\perp$ are distributed
according to this measure. So
\begin{equation}\label{eq:gauss_proj_embed_distr_equality}
\Ker(\bfG)\ \eqdistr\ \Ker(\bfH)^\perp,
\qquad\text{and this common distribution is uniform on } G(n,d'').
\end{equation}

The next proposition shows that one can couple two Gaussian projections $\bfG$ and $\bfH$
on a common probability space in such a way that the distributional identity
\eqref{eq:gauss_proj_embed_distr_equality} is realized by an almost sure equality of subspaces. ($\bfG$ and $\bfH$ may become dependent in this coupling.)

\begin{proposition}[Coupling of Gaussian projections]\label{prop:gauss_proj_gauss_embed_coupled}
Let $d',d''\in\mathbb N$ and set $n:=d'+d''$. There exists a probability space on which one can
define two Gaussian projections $\bfG:\mathbb R^{n}\to\mathbb R^{d'}$ and
$\bfH:\mathbb R^{n}\to\mathbb R^{d''}$ such that
\begin{equation}\label{eq:Ker_bfG_ker_bfH}
\Ker \bfG \;=\; (\Ker \bfH)^\bot \quad \text{a.s.}
\end{equation}
Moreover, if $e_1,\ldots,e_n$ denote the standard orthonormal basis of $\mathbb R^{n}$, then
for almost every realization of $(\bfG,\bfH)$ the vector configurations
\[
\bfG e_1,\dots,\bfG e_n \in \mathbb R^{d'}
\qquad\text{and}\qquad
\bfH e_1,\dots,\bfH e_n \in \mathbb R^{d''}
\]
are in linear Gale duality.
\end{proposition}

We shall give two proofs of Proposition~\ref{prop:gauss_proj_gauss_embed_coupled}. The first proof
is based on the following general coupling principle.

\begin{lemma}\label{lem:coupling_general}
Let $X\in\mathbb R^{p}$ and $Y\in\mathbb R^{q}$ be random vectors, and let
$\varphi:\mathbb R^{p}\to T$ and $\psi:\mathbb R^{q}\to T$ be Borel-measurable maps into a Polish space
$T$ (equipped with its Borel $\sigma$-algebra). Assume that $\varphi(X)$ and $\psi(Y)$ have the same
distribution on $T$. Then there exists a probability space carrying random vectors $X'$ and $Y'$
such that
\[
X'\eqdistr X,\qquad Y'\eqdistr Y,\qquad \text{and}\qquad \varphi(X')=\psi(Y') \ \text{a.s.}
\]
\end{lemma}

\begin{proof}
Let $(\pi_t^X)_{t\in T}$ be a regular conditional distribution of $X$ given $\varphi(X)$, and let
$(\pi_t^Y)_{t\in T}$ be a regular conditional distribution of $Y$ given $\psi(Y)$. Existence of these
kernels follows since $T$ is Polish; see, e.g.,~\cite[Theorem~10.2.2]{dudley_book_real_analysis_probab}.
For \(\mu\)-almost every \(t\in T\), where \(\mu\) is the common distribution of \(\varphi(X)\) and \(\psi(Y)\), 
\begin{equation}\label{eq:lemma_coupling_proof_supposrt_cond_law}
\pi_t^X\bigl(\varphi^{-1}(\{t\})\bigr)=1
\qquad\text{and}\qquad
\pi_t^Y\bigl(\psi^{-1}(\{t\})\bigr)=1.
\end{equation}

Now construct a random element $Z$ with values in $T$ such that
$Z\eqdistr \varphi(X)$ (equivalently, $Z\eqdistr \psi(Y)$). Conditionally on $Z=t$, sample $X'$
according to $\pi_t^X$ and $Y'$ according to $\pi_t^Y$ (for instance, take $X'$ and $Y'$ conditionally
independent given $Z$). Then $X'\eqdistr X$ and $Y'\eqdistr Y$ by the defining property of regular
conditional distributions. Moreover, the support properties~\eqref{eq:lemma_coupling_proof_supposrt_cond_law} imply $\varphi(X')=Z=\psi(Y')$ almost surely.
\end{proof}

\begin{proof}[First proof of Proposition~\ref{prop:gauss_proj_gauss_embed_coupled}]
We apply Lemma~\ref{lem:coupling_general} with $T:=G(n,d'')$,
$X \eqdistr \bfG$ and $Y \eqdistr \bfH$, viewed as random elements of
$\mathbb R^{d'\times n}$ and $\mathbb R^{d''\times n}$, respectively. Define Borel maps
\[
\varphi:\mathbb R^{d'\times n}\to T,\quad \varphi(M):=\Ker(M)
\quad \text{ and }\quad
\psi:\mathbb R^{d''\times n}\to T,\quad \psi(M):=\Ker(M)^\perp.
\]
By~\eqref{eq:gauss_proj_embed_distr_equality}, the random subspaces $\varphi(\bfG)=\Ker(\bfG)$ and
$\psi(\bfH)=\Ker(\bfH)^\perp$ have the same distribution on $T$. Hence Lemma~\ref{lem:coupling_general}
yields a coupling of $\bfG$ and $\bfH$ on a common probability space such that
\begin{equation}\label{eq:coupled-kernels}
\Ker(\bfG)=\Ker(\bfH)^\perp \quad\text{a.s.}
\end{equation}

The matrices with columns $\bfG e_1,\dots,\bfG e_n$ and $\bfH e_1,\dots,\bfH e_n$ are precisely $\bfG$ and $\bfH$. Since $\bfG$ and $\bfH$ are standard Gaussian matrices, they have full row rank almost surely,
i.e.\ $\rank(\bfG)=d'$ and $\rank(\bfH)=d''$ a.s.  Together with~\eqref{eq:coupled-kernels} this shows that  Conditions~\eqref{eq:Gale_duality_row_spaces_orthogonal_intro} are satisfied, hence the configurations
$\bfG e_1,\dots,\bfG e_n$ and $\bfH e_1,\dots,\bfH e_n$ are in linear Gale duality.
\end{proof}

\begin{remark}
The same argument proves the following more general claim. Consider random matrices $\bfA: \R^{d'+d''} \to \R^{d'}$ and $\bfB: \R^{d'+d''} \to \R^{d''}$ such that both $\bfA$ and $\bfB$ have full rank a.s.\ and $\Ker \bfA$  has the same distribution as $(\Ker \bfB)^\bot$ on the Grassmannian $G(d'+d'', d'')$.  Then, there is a coupling of $\bfA$ and $\bfB$ such that $\Ker \bfA = (\Ker \bfB)^\bot$ a.s. The proof follows from Lemma~\ref{lem:coupling_general}.
\end{remark}

\begin{proof}[Second proof of Proposition~\ref{prop:gauss_proj_gauss_embed_coupled}]
In this proof, we give an explicit construction of the coupling $(\bfG, \bfH)$.
Let $X_1,\ldots,X_{d''}$ be i.i.d.\ standard Gaussian points in $\R^{n} = \R^{d'+d''}$.  The linear Blaschke--Petkantschin formula~\cite[Theorem~7.2.1]{schneider_weil_book} states that for every nonnegative Borel function $h:(\R^{n})^{d''} \to [0,\infty)$,
\begin{multline*}
\E h(X_1,\ldots,  X_{d''}) = \int\displaylimits_{(\R^{n})^{d''}}  \frac{\eee^{-\frac 12 (\|z_1\|^2+ \ldots + \|z_{d''}\|^2)}}{(2\pi)^{n d''/2}} \cdot h(z_1,\ldots, z_{d''}) \; \prod_{i=1}^{d''} \dint z_i
\\
= B \int\displaylimits_{G(n, d'')}  \int\displaylimits_{L^{d''}}   \frac{\eee^{-\frac 12 (\|z_1\|^2+ \ldots + \|z_{d''}\|^2)}}{(2\pi)^{nd''/2}} \cdot  (\nabla_{d''}(z_1,\ldots, z_{d''}))^{d'} \cdot h(z_1,\ldots, z_{d''}) \; \prod_{i=1}^{d''} \lambda_L(\dint z_i)\; \nu_{n,d''}(\dint L),
\end{multline*}
where  $\nu_{n,d''}$ is the uniform probability distribution on $G(n, d'')$, $\lambda_L$ is the Lebesgue measure on $L\in G(n, d'')$, $\nabla_{d''}(z_1,\ldots, z_{d''})$ denotes the $d''$-dimensional volume of the parallelepiped spanned by the vectors $z_1,\ldots, z_{d''}$, and $B$ is a normalizing constant  whose value can be found in~\cite[Eqn.~(7.8), p.~271]{schneider_weil_book}.

To construct the coupling $(\bfG,\bfH)$, we consider the following two-stage random experiment. In a first stage, construct a random Gaussian matrix $\bfG: \R^{n} \to \R^{d'}$. In a second stage, conditionally on $\Ker \bfG = L$, where $L\subseteq \R^{n}$ is a $d''$-dimensional linear subspace, we construct random points $(Z_1,\ldots, Z_{d''})$ on $L$ whose joint probability density with respect to $(\lambda_L)^{\otimes d''}$ is given by
$$
f_L(z_1,\ldots, z_{d''})
:=
B \cdot  \frac{\eee^{-\frac 12 (\|z_1\|^2+ \ldots + \|z_{d''}\|^2)}}{(2\pi)^{n d''/2}} \cdot  (\nabla_{d''}(z_1,\ldots, z_{d''}))^{d'},
\qquad  z_1,\ldots, z_{d''}\in L.
$$
(The fact that $f_L$ integrates to $1$ for every fixed $L\in G(n, d'')$ follows from the Blaschke--Petkantschin formula with $h\equiv 1$ together with the fact that the integrand in $\int_{G(n, d'')} \ldots \nu_{n,d''}(\dint L)$ does not depend on $L$.) With probability $1$, the linear span of $Z_1,\ldots, Z_{d''}$ is $L$. Now, for every nonnegative Borel function $h:(\R^{n})^{d''} \to [0,\infty)$, the definition of $Z_1,\ldots, Z_{d''}$ yields
\begin{multline*}
\E h(Z_1,\ldots,  Z_{d''})
\\
= \int\displaylimits_{G(n, d'')}  \int\displaylimits_{L^{d''}}  B \cdot \frac{\eee^{-\frac 12 (\|z_1\|^2+ \ldots + \|z_{d''}\|^2)}}{(2\pi)^{n d''/2}} \cdot (\nabla_{d''}(z_1,\ldots, z_{d''}))^{d'}\cdot  h(z_1,\ldots, z_{d''}) \; \prod_{i=1}^{d''} \lambda_L(\dint z_i)\; \nu_{n,d''}(\dint L).
\end{multline*}
So $\E h(Z_1,\ldots,  Z_{d''}) = \E h(X_1,\ldots, X_{d''})$ for every $h$  and it follows that $(Z_1,\ldots, Z_{d''})$ have the same full joint distribution as $(X_1,\ldots, X_{d''})$, i.e.\ they are i.i.d.\ standard Gaussian points in $\R^{n}$. We can now construct $\bfH: \R^{n} \to \R^{d''}$ by declaring that $\bfH^\top$ maps  the standard orthonormal basis of $\R^{d''}$ to $Z_1,\ldots, Z_{d''}$.  Then, $\bfH$ is a Gaussian matrix and  $(\Ker \bfH)^\bot = \Im (\bfH^\top) = L = \Ker \bfG$.
\end{proof}

\subsection{Duality coupling for i.i.d.\ random unit vectors}
The next result will be repeatedly used  in the subsequent sections.

\begin{theorem}[Main duality coupling]\label{theo:gale_duality_stoch_orthant}
Fix some  $d,k\in \N$.  On a suitable probability space one can  construct random vectors $U_1,\ldots, U_{d+k}, V_1,\ldots, V_{d+k}$ such that
\begin{itemize}
\item[(a)]  $U_1,\ldots, U_{d+k}$ are independent and uniform on the unit sphere $\bS^{d-1}$ in $\R^{d}$;
\item[(b)]  $V_1,\ldots, V_{d+k}$ are independent and uniform on the unit sphere $\bS^{k-1}$ in $\R^{k}$;
\item[(c)]  For every outcome in the underlying probability space and every set $I\subseteq [d+k]$ with $\# I \leq d-1$ we have
\begin{equation}\label{eq:gale_duality_special_case_orthant}
\pos (U_i: i\in I) \text{ is a face of } \pos (U_1,\ldots, U_{d+k})
\quad \Longleftrightarrow  \quad
\pos (V_i: i\in I^c) = \R^{k}.
\end{equation}
Here $I^c := [d+k] \bsl I$ is the complement of the index set $I$.
\end{itemize}
\end{theorem}
\begin{proof}
Let $e_1,\ldots, e_{d+k}$ be the standard orthonormal basis of $\R^{d+k}$. By Proposition~\ref{prop:gauss_proj_gauss_embed_coupled}, we may couple Gaussian projections $\bfG: \R^{d+k} \to \R^d$  and $\bfH: \R^{d+k} \to \R^k$ such that, for almost every realization of $(\bfG, \bfH)$,   the vectors $\bfG e_1,\ldots, \bfG e_{d+k}\in \R^d$ are in linear Gale duality with the vectors $\bfH e_1,\ldots, \bfH e_{d+k}\in \R^k$. Moreover, with probability $1$, both configurations are in general linear position.  After restricting to a full-measure event we may assume that these properties hold for \emph{all} outcomes in the probability space.

We can therefore apply Lemma~\ref{lem:faces-positive-spanning-intro-version} to the Gale-dual
configurations $(\bfG e_i)_{i\in[d+k]}$ and $(\bfH e_i)_{i\in[d+k]}$. It yields that for every set $I\subseteq [d+k]$ with $\# I \leq d-1$,
\begin{equation}\label{eq:duality_G_h_tech_123}
\pos(\bfG e_i: i\in I) \text{ is a face of } \pos(\bfG e_i: i\in [d+k])
\quad \Longleftrightarrow \quad
\pos (\bfH e_i: i\in I^c) = \R^{k}.
\end{equation}
Now we define $U_{i} := \bfG e_i / \|\bfG e_i\|$ and $V_i := \bfH e_i/ \|\bfH e_i\|$ for all $i=1,\ldots, d+k$. (Note that $\bfG e_i\neq 0$ and $\bfH e_i \neq 0$ by the linear general position property.) Then  (a) and (b) are satisfied. Also, for every $M\subseteq [d+k]$ we have
$$
\pos(\bfG e_i: i\in M)  = \pos (U_i: i\in M),
\qquad
\pos (\bfH e_i: i \in M) = \pos (V_i: i\in M).
$$
Substituting these identities into
\eqref{eq:duality_G_h_tech_123} yields \eqref{eq:gale_duality_special_case_orthant}, completing the
proof.
\end{proof}
\begin{remark}\label{rem:oriented_matroid_duality}
The coupling in Theorem~\ref{theo:gale_duality_stoch_orthant} is naturally interpreted in the language of oriented matroids~\cite{bjoerner_etal_book_oriented_matroids}.
Indeed, for each outcome of the underlying probability space, the vector configurations
\((U_1,\ldots,U_{d+k})\subseteq\R^d\) and \((V_1,\ldots,V_{d+k})\subseteq\R^k\) determine oriented matroids on the common ground set \([d+k]\),
and the construction in the proof yields that these oriented matroids are dual to one another (in the sense of Gale duality).
In the present paper, however, we do not make systematic use of oriented matroid theory: the only consequence of this duality
that will be needed is the correspondence stated in part~(c) of Theorem~\ref{theo:gale_duality_stoch_orthant}.
\end{remark}

\section{Two examples}
Theorem~\ref{theo:gale_duality_stoch_orthant} will serve as a key tool throughout the remainder of the paper.
We start by illustrating its use with two simple applications.

\subsection{Independence of face events}\label{subsec:face_events_independence}
Consider the random cone  $\cW_{n,d} = \pos (U_1,\ldots, U_n)$, where $U_1,\ldots, U_n$ are independent and uniformly distributed on the unit sphere in $\R^d$.
For a subset $I\subseteq [n]$ define the cone $F_I = \pos (U_i: i\in I)$ and the random event
$$
A_I:= \{F_I \text{ is a face of } \cW_{n,d}\}.
$$
With probability $1$, the vectors $U_1,\ldots, U_n$ are in general linear position (assuming $d\geq 2$). Thus, for $\ell \in \{0,\ldots, d-1\}$, every $\ell$-dimensional face of $\cW_{n,d}$ has the form $F_I$ for some $I\subseteq[n]$ with $\# I = \ell$, but not every $F_I$ is necessarily a face of $\cW_{n,d}$.  In particular, the number of $\ell$-dimensional faces of $\cW_{n,d}$ can be represented as
$$
f_{\ell} (\cW_{n,d}) = \sum_{I\subseteq [n], \#I = \ell} \ind_{A_I}.
$$
Somewhat unexpectedly, many of the face events $A_I$ turn out to be independent; the next theorem
makes this precise.
\begin{theorem}[Independence of face events]
Take $d\geq 3$ and $n\geq d+1$.
Let $I_1,\ldots, I_r\subseteq [n]$ be such that $I_s\cup I_t = [n]$ for all $1\leq s\neq t\leq r$ and $\#I_s \leq d-1$ for all $s\in [r]$. Then, the face events $A_{I_1},\ldots, A_{I_r}$ are independent.
\end{theorem}
\begin{proof}
We use the coupling given in Theorem~\ref{theo:gale_duality_stoch_orthant} with $k=n-d$. Recall that the random vectors $V_1,\ldots, V_n$ appearing in that theorem are independent and uniformly distributed on the unit sphere in $\R^{n-d}$. For $J\subseteq [n]$ consider the cone $D_J = \pos (V_j: j\in J)\subseteq \R^{n-d}$.  By Theorem~\ref{theo:gale_duality_stoch_orthant},
\begin{equation}\label{eq:face_event_gale_dual}
A_I= \{F_I \text{ is a face of } \cW_{n,d}\} = \{D_{I^c}= \R^{n-d}\} \qquad  \text{ for all } I \subseteq [n] \text{ with } \#I \leq d-1.
\end{equation}
Under our assumptions, the index sets $I_1^c,\ldots, I_r^c$ are disjoint. Hence the random events $\{D_{I_1^c} = \R^{n-d}\},\ldots, \{D_{I_r^c} = \R^{n-d}\}$ are independent.
\end{proof}

\subsection{Expected face counts}
As a simple consequence of  Theorem~\ref{theo:gale_duality_stoch_orthant}, we recover the following  result of \citet[Theorem~1.6]{donoho_tanner_orthants}. Recall that Wendel's formula~\cite{wendel} (see also~\cite[Theorem~8.2.1]{schneider_weil_book})  states that
\begin{equation}\label{eq:wendel_formula_repeat}
p(n,d) = \P[\cW_{n,d} = \R^d] =  \frac 1  {2^{n-1}} \sum_{\ell=d}^{n-1} \binom{n-1}{\ell}.
\end{equation}

\begin{proposition}[Expected face numbers]\label{prop:expected_face_number_donoho_tanner_cone}
For every  $d\geq 2$, $n\geq d+1$ and $\ell \in \{0,\ldots, d-1\}$, the expected number of $\ell$-dimensional faces of the random cone $\cW_{n,d}$ is
$$
\E f_\ell (\cW_{n,d}) = \binom{n}{\ell} p(n-\ell, n-d) = \frac{1}{2^{n-\ell-1}}\binom{n}{\ell}\sum_{r=n-d}^{n-\ell-1}\binom{n-\ell-1}{r}.
$$
Also, for every $I\subseteq [n]$ with $\# I= \ell \in \{0,\ldots, d-1\}$, we have $\P[A_I] = p(n-\ell, n-d)$.
\end{proposition}
\begin{proof}
For every $I$ as above, Formula~\eqref{eq:face_event_gale_dual} gives
$$
\P[A_I] = \P[\pos (V_i: i\in I^c) = \R^{n-d}] = \P[\cW_{n-\ell,n-d} = \R^{n-d}] =  p(n-\ell, n-d),
$$
where the last step follows from Wendel's formula~\eqref{eq:wendel_formula_repeat}.
Taking the sum over all subsets $I \subseteq [n]$ with $\#I = \ell$    gives the stated formula for $\E f_\ell (\cW_{n,d})$.
\end{proof}

\begin{remark}
Taking up a suggestion of Gale, \citet{Schneider2021RandomGale} proposed to study random polytopes $P$ obtained by choosing their (affine) Gale transform at random. (Throughout this remark, we use the term ``polytope'' to refer to its \emph{combinatorial type}; only the combinatorial type is canonically determined by the Gale transform.)  The model of \citet[p.~643]{Schneider2021RandomGale} depends on an even probability measure $\phi$; in the present remark we take $\phi$ to be the uniform distribution on the unit sphere in $\R^k$. We put $n= d+k$.
Using the notation of Theorem~\ref{theo:gale_duality_stoch_orthant}, \citet{Schneider2021RandomGale} considers a $(d-1)$-dimensional random polytope $P$ generated by $n$ points whose (affine) Gale transform is given by $\lambda_1 V_1,\ldots,\lambda_n V_{n}\in \R^k$ (with suitable $\lambda_j >0$) conditioned on the event $\pos(V_1,\ldots,V_{n}) = \R^{k}$.
By Example~\ref{ex:I-empty}, this is the same conditioning event as $\pos(U_1,\ldots,U_n)\neq \R^d$.
It then follows from Theorem~\ref{theo:gale_duality_stoch_orthant} that the random cone $\cW_{n,d}=\pos(U_1,\ldots,U_n)$ conditioned on the event $\{\cW_{n,d}\neq \R^d\}$, can be identified with the cone over $P$. Conversely, $P$ can be recovered as a bounded slice (cross-section) of this conditioned random cone. Such conditioned cones are called \emph{Cover--Efron cones} in~\cite{hug_schneider_conical_tessellations}.
Consequently, the random polytopes studied in~\citet{Schneider2021RandomGale} have the same distribution as the bounded cross-sections of Cover--Efron cones considered in~\cite{hug_schneider_conical_tessellations}. (As noted in~\cite[Remark on p.~648]{Schneider2021RandomGale}, these two objects in particular have the same expected face numbers.)
\end{remark}

\section{Moments of random angles}\label{sec:moments_angles}

\subsection{Duality for moments of random angles}
The \emph{solid angle} of a polyhedral cone with non-empty interior $C\subseteq \R^d$  is defined as
$
\alpha(C) = \nu_d(C\cap \bS^{d-1}) \in [0,1],
$ where $\nu_d$ is the uniform probability distribution on $\bS^{d-1}$.

Let $U_1,\ldots, U_d$ be i.i.d.\ random unit vectors uniformly distributed on the unit sphere in $\R^d$, $d\in \N$. The solid angle of the cone $\cW_{d,d} =\pos (U_1,\ldots, U_d)$  is a random variable $\alpha_d := \alpha (\pos (U_1,\ldots, U_d))$ taking values in the interval $[0, 1/2]$. 
We are interested in the moments of $\alpha_d$, which we denote by
$$
m(d,k):= \E [\alpha_d^k], \qquad d,k\in \N.
$$
Clearly,  $m(d,1) = \E \alpha_d = 2^{-d}$.  (Indeed, the cones $\pos (\eps_1 U_1,\ldots, \eps_d U_d)$, with $\eps_j\in \{\pm 1\}$, have disjoint interiors, cover $\R^d$ and have the same expected angle.)

\begin{theorem}[Duality for random solid angles]\label{theo:duality_moments_angles}
For all $d,k\in \N$ we have $m(d,k) = m(k,d)$.
\end{theorem}
\begin{proof}
Let $U_1,\ldots, U_d, U_{d+1}, \ldots, U_{d+k}$ be i.i.d.\ random unit vectors with the uniform distribution on $\bS^{d-1}$.
We shall compute the probability of the event
$U_{d+1},\ldots, U_{d+k} \in \pos (U_1,\ldots, U_d)$
in two different ways.

\vspace*{2mm}
\noindent
\textit{Primal method}.
Conditioning on $U_1,\ldots,  U_d$, we have, by definition of the solid angle,
$$
\P[U_{d+1},\ldots, U_{d+k} \in \pos (U_1,\ldots, U_d) \,|\, U_1,\ldots,  U_d] =  \left(\alpha (\pos (U_1,\ldots, U_d))\right)^k.
$$
Integrating over $U_1,\ldots, U_d$ we arrive at
\begin{align*}
\P[U_{d+1},\ldots, U_{d+k} \in \pos (U_1,\ldots, U_d)]
=
\E [(\alpha (\pos(U_1,\ldots, U_d)))^k]
=
m(d,k).
\end{align*}

\vspace*{2mm}
\noindent
\textit{Dual method.}
By Theorem~\ref{theo:gale_duality_stoch_orthant}, we may realize $(U_1,\ldots,U_{d+k})$ together
with $(V_1,\ldots,V_{d+k})$ on a common probability space such that the duality
\eqref{eq:gale_duality_special_case_orthant} holds. In the following we agree to ignore null events.
On the event that $\pos(U_1,\ldots,U_d)\cap \bS^{d-1}$ is a spherical simplex (which holds a.s.), the
condition $U_{d+1},\ldots,U_{d+k}\in\pos(U_1,\ldots,U_d)$ is equivalent to requiring that all
facets of $\pos(U_1,\ldots,U_d)$ remain facets of $\pos(U_1,\ldots,U_{d+k})$. Hence,
\begin{align*}
\lefteqn{\P\Bigl[U_{d+1},\ldots, U_{d+k} \in \pos (U_1,\ldots, U_d)\Bigr]}
\\
&=
\P\Bigl[\forall \ell\in[d]:\ \pos(U_i:i\in[d]\setminus\{\ell\}) \text{ is a facet of }
\pos(U_1,\ldots, U_{d+k})\Bigr]\\
&=
\P\Bigl[\forall \ell\in[d]:\ \pos(V_i:i\in([d]\setminus\{\ell\})^{c})=\R^{k}\Bigr],
\end{align*}
where we used~\eqref{eq:gale_duality_special_case_orthant}. Since
$([d]\setminus\{\ell\})^{c}=\{\ell,d+1,\ldots,d+k\}$, this becomes
\[
\P[U_{d+1},\ldots, U_{d+k} \in \pos (U_1,\ldots, U_d)]
=
\P\Bigl[\forall \ell\in[d]:\ \pos(V_\ell,V_{d+1},\ldots,V_{d+k})=\R^{k}\Bigr].
\]
Now, we observe that $\pos (V_\ell, V_{d+1},\ldots, V_{d+k}) = \R^k$ if and only if $V_\ell \in -\Int \pos ( V_{d+1},\ldots, V_{d+k})$,  ignoring null events.  Thus,
\begin{align*}
\P[U_{d+1},\ldots, U_{d+k} \in \pos (U_1,\ldots, U_d)]
&=
\P[V_1,\ldots, V_d \in -\pos ( V_{d+1},\ldots, V_{d+k})]
\\
&=
\E [(\alpha (\pos(V_{d+1},\ldots, V_{d+k})))^d]
\\
&=
m(k,d).
\end{align*}
Comparing the results obtained by both methods, we conclude that $m(d,k) = m(k,d)$.
\end{proof}

\subsection{Moments of small order}
Theorem~\ref{theo:duality_moments_angles} allows us to find explicit formulas for the first three moments of $\alpha_d$.
\begin{theorem}[Moments of random angles]\label{theo:random_angle_three_moments}
For every $d\in \N$, the first three moments of the random variable $\alpha_d := \alpha (\pos (U_1,\ldots, U_d))$ are given by
\begin{align}
&\E [\alpha_d] = m(d,1)  = m(1,d) =  2^{-d},
\\
&\E [\alpha_d^2]  = m(d,2) =  m(2,d) =  \frac{1}{2^d (d+1)}, \label{eq:moments_angle_2}
\qquad
\Var [\alpha_d]  = \frac{1}{2^d (d+1)} - \frac 1 {4^d},
\\
& \E [\alpha^3_d ] = m(d,3) = m(3,d) = \frac 1 {(4\pi)^d}\E [S^d] =  \frac 1 {(4\pi)^d} \int_0^{2\pi}  x^d f_S(x)\dd x, \label{eq:moments_angle_3}
\end{align}
where $S$ is a random variable with probability density
\begin{align*}
f_S(x)
=
-\frac{(x^2 - 4\pi x + 3\pi^2 - 6)\cos x
       - 6(x - 2\pi)\sin x
       - 2(x^2 - 4\pi x + 3\pi^2 + 3)}
      {16\pi\,\cos^4(x/2)},
\qquad
0 \leq  x \leq  2\pi.
\end{align*}
\end{theorem}
\begin{proof}
To give an alternative proof that $\E[\alpha_d] =2^{-d}$, observe that $m(1,d)= 2^{-d}$  since  $V_1$ is uniform on $\{\pm 1\}$ and hence $\pos (V_1)$ is a ray with angle $1/2$.

To prove~\eqref{eq:moments_angle_2},  it suffices to compute $m(2,d)$.  If $V_1, V_2$ are independent and uniformly distributed on the unit circle in $\R^2$, then the random angle $\alpha (\pos(V_1,V_2))$ is uniformly distributed on $[0,1/2]$. The $d$-th moment of this random variable is $2^{-d}\cdot 1/(d+1)$, and the claim follows.

Finally, the formula for $m(3,d)$ follows from the following non-trivial fact due to Crofton and Exhumatus~\cite{CroftonExhumatus1867}: The spherical area $S$ of the random spherical triangle with vertices $V_1,V_2,V_3$ that are independent and uniform on $\bS^2$ has probability density $f_S$ given as above. Proofs can be found in~\cite{philip_random_spherical_triangle_2014} and~\cite{finch_jones_random_spherical_triangles_2010}; see also~\cite{SemiclassicalRavenclawPrefect2021} and~\cite{miles_random_points_sphere}. Then $\alpha(\pos (V_1,V_2,V_3)) = S/(4\pi)$ and $m(3,d) = \E [S^d/(4\pi)^d]$.
\end{proof}

\begin{remark}[The third moment]\label{rem:third_moment_angle}
Explicit formulas for $\E [S^d]$  with $d\leq 10$ are given in~\cite{SemiclassicalRavenclawPrefect2021}. (In the formula for $\E[S^4]$, a factor of $108$ is missing in front of $\zeta(3)$.)  These formulas give the trivial values $\E[\alpha_1^3] = \frac{1}{8}$, $\E[\alpha_2^3]= \frac{1}{32}$ (since $\alpha_2$ is uniform on $[0,1/2]$) as well as
$$
\E[\alpha_3^3]
= \frac{3}{128} \;-\; \frac{3\log 2}{16\pi^2},
\qquad
\E[\alpha_4^3]
= \frac{\pi^4 - 6\pi^2\log 2 - 27\,\zeta(3)}{64\,\pi^4},
  \qquad
\E[\alpha_5^3]
= \frac{5}{512} \;-\; \frac{15\log 2}{128\pi^2}.
$$
%Chat GPT: Using the substitution $t= \tan (x/2)$, the integral in~\eqref{eq:moments_angle_3} can be reduced to a sum of the integrals of the form $\int_{-\infty}^\infty t^\ell \arctan^k x$. It can be shown that $\E[S^d]$ is a linear combination (with rational coefficients) of $\pi^d$, $\pi^{d-2} \log 2$ and $\pi^{d-2m-1}\zeta (2m+1)$, $m=1,\ldots, \lfloor d/2\rfloor-1$.
\end{remark}

\begin{remark}[Asymptotics of the third moments]\label{rem:asymptotic_third_moment_solid_angle}
Since $f_S(2\pi) = \frac{12 - \pi^2}{16\pi}>0$ and $f_S(x)$ is continuous at $x= 2\pi$, the standard Laplace asymptotics gives
\[
\E [\alpha^3_d ] = \frac 1 {(4\pi)^d} \E [S^d]
\sim
\frac 1 {(4\pi)^d} \frac{f_S(2\pi)\,(2\pi)^{d+1}}{d+1}
\sim
\frac{12 - \pi^2}{8}
\frac{1}{2^d d},
\qquad
d\to\infty.
\]
\end{remark}

\begin{remark}
In this section, we considered the cone $\cW_{n,d}$ with $n=d$ only. Our methods do not allow us to compute the variance of $\alpha (\cW_{n,d})$ for $n>d$. What we can show is that for $n\geq d$,
\begin{equation}\label{eq:duality_second_moment_angle_f_1}
\E [\alpha^2 (\cW_{n,d})] = \frac {\E [f_1^2(\cW_{n+2, n+2-d})] - \E [f_1(\cW_{n+2, n+2-d})]} {(n+2)(n+1)}.
\end{equation}
\end{remark}

\begin{proof}[Proof of~\eqref{eq:duality_second_moment_angle_f_1}]
Let $k\in \N$.
By Theorem~\ref{theo:gale_duality_stoch_orthant}, we may realize i.i.d.\ random vectors $U_1,\ldots,U_{n+k}\sim \mathrm{Unif}(\mathbb S^{d-1})$ together with the i.i.d.\ random vectors $V_1,\ldots,V_{n+k}\sim \mathrm{Unif}(\mathbb S^{n+k-d-1})$ on a common probability space such that the duality~\eqref{eq:gale_duality_special_case_orthant}
 holds.  We have
\begin{align*}
\E [\alpha^k(\cW_{n,d})]
&=
\P \left[
U_{n+1},\ldots,U_{n+k}\in -\pos(U_1,\ldots,U_n)
\right]
\\
&=
\P \left[\pos (U_1,\ldots,U_n,U_{n+i})
=
\mathbb R^d
\ \forall i=1,\ldots,k
\right].
\end{align*}
Applying the duality~\eqref{eq:gale_duality_special_case_orthant}, with $U_i$'s and $V_j$'s interchanged, to the right-hand side gives
\[
\E [\alpha^k(\cW_{n,d})]
=
\P\left[
\pos\bigl(\{V_{n+1},\ldots,V_{n+k}\}\setminus\{V_{n+i}\}\bigr)
\text{ is a face of }
\operatorname{pos}(V_1,\ldots,V_{n+k})
\ \forall i=1,\ldots,k
\right].
\]
If \(k=2\), the condition on the  right-hand side states that  $\operatorname{pos}(V_{n+1})$ and $\operatorname{pos}(V_{n+2})$
are faces of $\operatorname{pos}(V_1,\ldots,V_{n+2})$. So, using exchangeability,
\begin{align*}
\E [\alpha^2(\cW_{n,d})]
&=
\P \left[
\pos(V_{n+1}), \pos(V_{n+2})\in
\mathcal F_1 \left(
\pos(V_1,\ldots,V_{n+2})
\right)
\right]
\\
&=
\frac{1}{(n+2)(n+1)}
\,
\E
\!
\sum_{\substack{i\neq j\\ i,j\in\{1,\ldots,n+2\}}}
\ind_{\{\pos(V_i) \in \mathcal F_1(\pos(V_1,\ldots,V_{n+2}))\}}
\,
\ind_{\{\pos(V_j)\in \mathcal F_1(\pos(V_1,\ldots,V_{n+2}))\}}
\\
&=
\frac{1}{(n+2)(n+1)}
\,
\mathbb E\!\left[
f_1^2\left(\pos(V_1,\ldots,V_{n+2})\right)
-
f_1\left(\pos(V_1,\ldots,V_{n+2})\right)
\right],
\end{align*}
and the proof is complete since $\pos(V_1,\ldots,V_{n+2})$ has the same law as $\cW_{n+2, n+2-d}$.
\end{proof}

\subsection{Conjectures on high-dimensional random angles}\label{subsec:conjectures_high_dim_angles}
It seems difficult to compute $m(d, k)= m(k,d)$ with $k,d\geq 4$.  Here we state a conjecture on the asymptotics of $m(d,k)$ as $d\to\infty$ with fixed $k\geq 2$. Let $V_1,\ldots, V_k$ be independent and uniform on $\bS^{k-1}$.  The maximal possible value of the random variable $\alpha(\pos (V_1,\ldots, V_k))$ is $1/2$. An analysis of configurations of $k$ unit vectors in $\R^k$ close to being linearly dependent suggests that the density of $\alpha(\pos (V_1,\ldots, V_k))$ at $1/2$ should be positive, which (arguing as in Remark~\ref{rem:asymptotic_third_moment_solid_angle})  leads to the following conjecture.

\begin{conjecture}[Asymptotics for higher moments]\label{conj:asymptotics_moments_angle}
For every fixed $k\in \{2,3,\ldots\}$, we have $m(d,k) \sim c_k/ (2^d \cdot d)$ as $d\to\infty$, where $c_k>0$ is a constant. For example, $c_2 = 1$ and $c_3 = (12-\pi^2)/8$.  (Note that the case $k=1$ is different since $m(d,1) = 2^{-d}$.)
\end{conjecture}

These moment asymptotics suggest the following
conjecture: the measures
$\mu_d(A):=d\,2^d\,\P[\alpha_d\in A]$,
where $A\subseteq(0,\infty)$ is Borel,
converge vaguely on $(0,\infty)$ to a
nonzero Radon measure.

\section{A spherical Sylvester problem}\label{sec:spherical_sylvester}
Take some $d\geq 2$, $k\in \N$, and let $U_1,\ldots, U_{d+k}$ be i.i.d.\ random unit vectors uniformly distributed on the unit sphere in $\R^d$.  We consider the random spherical polytope
$$
\cP_{d+k,d} := \cW_{d+k,d}\cap \bS^{d-1} = \pos (U_1,\ldots, U_{d+k}) \cap \bS^{d-1}.
$$
Wendel's formula~\eqref{eq:wendel_formula} gives the probability that $\cP_{d+k,d}$ is the whole sphere.
\begin{proposition}[Probability of spherical simplex]\label{prop:probab_spherical_simplex}
For $d\geq 2$, $k\in \N$, the probability that $\cP_{d+k,d}$
is a spherical simplex (with $d$ vertices) equals
$$
\binom{d+k}{d} m(d,k) = \binom{d+k}{k} m(k,d).
$$
\end{proposition}
\begin{proof}
The event $\{\cP_{d+k,d} \text{ is a spherical simplex}\}$ occurs if and only if some (necessarily unique, for $d\geq 2$) set of $k$ vectors $U_{i_1}, \ldots, U_{i_k}$ is contained in the positive hull of the remaining $d$ vectors. By exchangeability, it follows that
\begin{align*}
\P[\cP_{d+k,d} \text{ is a spherical simplex}]
&=
\binom{d+k}{d} \cdot \P[U_{d+1},\ldots, U_{d+k} \in \pos (U_1,\ldots, U_d)]\\
&=
\binom{d+k}{d} \cdot \E [(\alpha (\pos(U_1,\ldots, U_d)))^k]
=
\binom{d+k}{d} m(d,k),
\end{align*}
and the proof is complete since $m(d,k) = m(k,d)$ by Theorem~\ref{theo:duality_moments_angles}.
\end{proof}

\begin{example}\label{exam:probab_spherical_simplex}
Taking $k=1,2,3$ and recalling Theorem~\ref{theo:random_angle_three_moments} yields the following results:
\begin{align}
&\P[\cP_{d+1,d} \text{ is a spherical simplex}] = \frac{d+1}{2^d}, \label{eq:probab_spher_simplex_k=1}
\\
&\P[\cP_{d+2,d} \text{ is a spherical simplex}] = \frac{d+2}{2^{d+1}}, \label{eq:probab_spher_simplex_k=2}
\\
&\P[\cP_{d+3,d} \text{ is a spherical simplex}] = \frac 1 {(4\pi)^d} \binom{d+3}{3} \int_0^{2\pi}  x^d f_S(x)\dd x,  \label{eq:probab_spher_simplex_k=3}
\end{align}
where $f_S(x)$ is as in Theorem~\ref{theo:random_angle_three_moments}. Equation~\eqref{eq:probab_spher_simplex_k=1} recovers a formula  obtained in~\cite[Lemma~4.4]{maehara_martini_sylvester_sphere}; see also~\cite[Theorem~3.6]{maehara_martini_geometric_probab}, \cite[Corollary~8.1]{maehara_martini_book}, \cite[Section~4.2]{kabluchko_panzo}. Note that~\eqref{eq:probab_spher_simplex_k=1} admits a simple proof:  $\cP_{d+1,d}$ is a spherical simplex if one of the vectors ($d+1$ choices) is inside the positive hull of the remaining vectors, and the probability that $U_{d+1} \in \pos (U_1,\ldots, U_d)$ is the expected angle of $\pos (U_1,\ldots, U_d)$, which is equal to $2^{-d}$ for symmetry reasons (see~\cite[Section~4]{maehara_martini_sylvester_sphere} for details).
On the other hand, Equations~\eqref{eq:probab_spher_simplex_k=2} and~\eqref{eq:probab_spher_simplex_k=3} seem to be new. For example, \eqref{eq:probab_spher_simplex_k=3} combined with Remark~\ref{rem:third_moment_angle} gives
\begin{align*}
&\P[\cP_{6,3} \text{ is a spherical simplex}] = \frac {15}{32} - \frac{15 \log 2}{4\pi^2}\approx 0.205385,
\\
&\P[\cP_{7,4} \text{ is a spherical simplex}] = \frac{35}{64}-\frac{105\log 2}{32\pi^2}-\frac{945\,\zeta(3)}{64\pi^4} \approx 0.134219,
\\
&\P[\cP_{8,5} \text{ is a spherical simplex}]  = \frac{35}{64} \;-\; \frac{105\,\log 2}{16\pi^2} \approx 0.085987.
\end{align*}
\end{example}

Next we consider the case $n=d+2$ in more detail. The number of vertices of the spherical polytope $\cP_{d+2,d}$ can take values $d$ (if $\cP_{d+2,d}$ is a spherical simplex),  $d+1$, $d+2$ and $0$ (equivalently, $\cP_{d+2,d} = \bS^{d-1}$).  In the next theorem we compute the probabilities of these possibilities.

\begin{theorem}[$d+2$ points on $\bS^{d-1}$]\label{theo:d+2_points_on_sphere_in_Rd}
For all $d\geq 2$  we have
\begin{align}
&\P[\cP_{d+2,d} = \bS^{d-1}] = \frac{d+2}{2^{d+1}}, \label{eq:probab_d+2_points_sphere_absorb}
\\
&\P[\cP_{d+2,d} \text{ is a spherical simplex}] = \frac{d+2}{2^{d+1}}, \label{eq:probab_d+2_points_sphere_simplex}
\\
&\P[\cP_{d+2,d} \text{ has } d+1 \text{ vertices}] = \frac{(d-2)(d+2)}{2^{d+1}},  \label{eq:probab_d+2_points_sphere_d+1_vertex}
\\
&\P[\cP_{d+2,d} \text{ has } d+2  \text{ vertices}] = 1 - \frac{d(d+2)}{2^{d+1}}.  \label{eq:probab_d+2_points_sphere_convex_position}
\end{align}
\end{theorem}
\begin{proof}
By Wendel's formula~\eqref{eq:wendel_formula_repeat}, $\P[\cP_{d+2,d} = \bS^{d-1}] = p(d+2,d) = (d+2)/2^{d+1}$, which proves~\eqref{eq:probab_d+2_points_sphere_absorb}.  We already proved~\eqref{eq:probab_d+2_points_sphere_simplex} in Example~\ref{exam:probab_spherical_simplex}. It remains to prove~\eqref{eq:probab_d+2_points_sphere_d+1_vertex} and~\eqref{eq:probab_d+2_points_sphere_convex_position}.  By Proposition~\ref{prop:expected_face_number_donoho_tanner_cone} with $\ell=1$ (taking into account that $f_0(\cP_{d+2,d}) = f_1 (\cW_{d+2,d})$),
$$
\E [d+2 - f_0(\cP_{d+2,d})] = (d+2) (1 - p(d+1,2)) = (d+2)p(d+1,d-1) = \frac{(d+2)(2d+2)}{2^{d+1}},
$$
where we again used Wendel's formula~\eqref{eq:wendel_formula_repeat}.
On the other hand, by definition of expectation and~\eqref{eq:probab_d+2_points_sphere_simplex},
\begin{multline*}
\E [d+2 - f_0(\cP_{d+2,d})] = (d+2)\cdot \P[\cP_{d+2,d} = \bS^{d-1}] + 2\cdot \P[\cP_{d+2,d} \text{ is a simplex}] \\
+ 1\cdot \P[\cP_{d+2,d} \text{ has $d+1$ vertices}]
=\frac{(d+2) (d+2 + 2)}{2^{d+1}}+ \P[\cP_{d+2,d} \text{ has $d+1$ vertices}].
\end{multline*}
Comparing both results gives~\eqref{eq:probab_d+2_points_sphere_d+1_vertex}. To prove~\eqref{eq:probab_d+2_points_sphere_convex_position}, observe that the complement of $\{\cP_{d+2,d} \text{ has } d+2 \text{ vertices}\}$ is the disjoint union of the events appearing in~\eqref{eq:probab_d+2_points_sphere_absorb}, \eqref{eq:probab_d+2_points_sphere_simplex}, \eqref{eq:probab_d+2_points_sphere_d+1_vertex}.
\end{proof}

\begin{example}[$5$ points on $\bS^2$]
Taking $d=3$ in Theorem~\ref{theo:d+2_points_on_sphere_in_Rd} gives:
The probability that the spherical convex hull of five random points on $\bS^2$, chosen uniformly and independently,  is a spherical triangle, quadrilateral, or pentagon is $5/16, 5/16, 1/16$, respectively. (Also, with probability $5/16$, their positive hull is the whole of $\R^3$).  Elementary proofs  can be found in~\cite[Theorem 3.6]{maehara_martini_geometric_probab}, \cite[Exercises~8.1, 8.2 on p.~170]{maehara_martini_book}, \cite{SemiclassicalRavenclawPrefect2021}.
\end{example}

Let us now consider the regime where $d\to\infty$ while $k\ge 2$ is fixed. Conjecture~\ref{conj:asymptotics_moments_angle}, together with Proposition~\ref{prop:probab_spherical_simplex}, would imply that the probability that $\cP_{d+k,d}$ is a spherical simplex is asymptotically equivalent to $(c_k/k!)\, d^{k-1} 2^{-d}$ for some constant $c_k>0$. Section~\ref{subsec:conjectures_high_dim_angles} suggests the following mechanism by which this rare event may occur: a subset of $d$ vectors $U_{i_1},\ldots,U_{i_d}$ spans a cone with solid angle of constant order, and the remaining $k$ vectors fall inside this cone. This leads to the following conjecture.

\begin{conjecture}
Fix $k\ge 2$. The conditional distribution of $\alpha(\cW_{d+k,d})$, given that $\cW_{d+k,d}\cap \bS^{d-1}$ is a spherical simplex, converges weakly (as $d\to\infty$) to a limit law not concentrated at $0$.
\end{conjecture}

Let us record one further consequence of Theorem~\ref{theo:gale_duality_stoch_orthant}.

\begin{proposition}[Probability of \(\ell\)-neighborliness]\label{prop:prob_l_neighborly}
Let $d,k\in \N$. Let $U_1,\ldots,U_{d+k}$ be independent and uniform on $\bS^{d-1}$ and let  $V_1,\ldots, V_{d+k}$ be independent and uniform on $\bS^{k-1}$.
For every \(\ell\in\{0,\ldots,d-1\}\), the following events have the same probability:
\begin{itemize}
\item[(a)] \(f_\ell(\cW_{d+k,d}) = \binom{d+k}{\ell}\), that is, for every subset \(I\subseteq[d+k]\) with \(\#I=\ell\), the cone \(\pos(U_i:i\in I)\) is an \(\ell\)-dimensional face of \(\cW_{d+k,d}= \pos (U_1,\ldots, U_{d+k})\).
\item[(b)] For every \(J\subseteq[d+k]\) with \(\#J=d+k-\ell\), one has
\(\bigcap_{j\in J}\mathrm{HS}_j=\varnothing\),
where \(\mathrm{HS}_j:=\{x\in\bS^{k-1}:\langle x,V_j\rangle\le 0\}\), \(j=1,\ldots,d+k\).
Equivalently, the family of half-spheres \((\mathrm{HS}_j)_{j=1}^{d+k}\) covers every point of \(\bS^{k-1}\) at most \(d+k-\ell-1\) times.
\item[(c)] The family \((\mathrm{HS}_j)_{j=1}^{d+k}\) covers every point of \(\bS^{k-1}\) at least \(\ell+1\) times.
\end{itemize}
\end{proposition}

\begin{proof}
By Theorem~\ref{theo:gale_duality_stoch_orthant}\,(c), the event in~(a) has the same probability as the event that for every
\(J\subseteq[d+k]\) with \(\#J=d+k-\ell\) one has \(\pos(V_j:j\in J)=\R^k\).
By polarity, \(\pos(V_j:j\in J)\ne\R^k\) holds if and only if there exists \(x\in\bS^{k-1}\) with
\(\langle x,V_j\rangle\le 0\) for all \(j\in J\), that is, if and only if \(\bigcap_{j\in J}\mathrm{HS}_j\neq\varnothing\).
Hence \(\pos(V_j:j\in J)=\R^k\) is equivalent to \(\bigcap_{j\in J}\mathrm{HS}_j=\varnothing\), which is precisely~(b).

To see that the events in (b) and (c) have equal probabilities, observe that
\begin{align*}
\text{Event in (b)}
&\;\Longleftrightarrow\;
\forall x\in\bS^{k-1}: \#\{j\in[d+k]: x\in \mathrm{HS}_j\}\le d+k-\ell-1
\\
&\;\Longleftrightarrow\;
\forall x\in\bS^{k-1}:  \#\{j\in[d+k]: x\in \bS^{k-1}\bsl \mathrm{HS}_j\}\ge \ell+1.
\end{align*}
Since $\bS^{k-1}\bsl \mathrm{HS}_j=\{z\in \bS^{k-1}:\langle z,-V_j\rangle < 0\}$ and since \((V_1,\ldots,V_{d+k})\) and \((-V_1,\ldots,-V_{d+k})\) have the same distribution, the probability of the latter event equals the probability of (c).
\end{proof}

\section{Limit theorems for face counts} \label{sec:limit_theorems_random_cones_high_dim}

\subsection{Statement of the limit theorem}
Recall that we consider the random polyhedral cone $\cW_{n,d}=\pos(U_1,\ldots,U_n)$, where
$U_1,\ldots,U_n$ are independent and uniformly distributed on the unit sphere in $\R^d$. Not much seems to be known about limit distributions of $f_{\ell}(\cW_{n,d})$ as $n,d,\ell \to\infty$.
We now describe the asymptotic fluctuations of the $(d-q)$-dimensional face count of
$\cW_{d+k,d}$ in the regime $d\to\infty$ with $k\in \N$ and $q\in \N$ fixed. Recall from Proposition~\ref{prop:expected_face_number_donoho_tanner_cone} the explicit formula
$$
\E f_{d-q} (\cW_{d+k,d}) = \binom{d+k}{d-q} p(k+q, k) = \frac{1}{2^{k+q-1}}\binom{d+k}{d-q}\sum_{r=k}^{k+q-1}\binom{k+q-1}{r},
\qquad d\geq q.
$$

\begin{theorem}[Distributional limit for face counts]\label{theo:distributional_limit_face_count_random_cone}
Let $k\in \N$ and $q \in \N$ be fixed. Then, as $d\to\infty$, the number of $(d-q)$-dimensional faces of $\cW_{d+k,d}$ satisfies
\begin{equation}
d\cdot \left(\frac{f_{d-q}(\cW_{d+k,d})}{\binom {d+k}{d-q}} - p(k+q, k)\right) \;{\overset{w}{\underset{d\to\infty}\longrightarrow}}\;
\frac {\binom{k+q}{2}\binom {k+q-2}{k-1}} {2^{k+q-1}}  \, (1 -Q_k),
\end{equation}
where $Q_k > 0$ is a random variable admitting the distributional representation $Q_1 \eqdistr  \operatorname{Gamma}(\frac 12, \frac 12)$ (for $k=1$) and
\begin{equation}\label{eq:Q_k_distributional}
Q_k \eqdistr \sum_{n=1,3,5,\ldots}  \left(\frac{\Gamma\!\left(\frac{k}{2}\right)\Gamma\!\left(\frac{n}{2}\right)}
{\pi\,\Gamma\!\left(\frac{n+k}{2}\right)}\right)^2 \operatorname{Gamma} \left(\frac{d_{n,k}}{2}, \frac 12\right),
\qquad \text{ for }k\geq 2.
\end{equation}
Here $d_{n,k} := \binom{n+k-1}{n}-\binom{n+k-3}{n-2}$ with $k\geq 2$, $n\in \N$ is the dimension of the space of degree $n$ spherical harmonics on $\bS^{k-1}$ (see Section~\ref{subsec:proof_theo_f_vect_U_stat_part_2}), and  $\operatorname{Gamma}(\alpha, \lambda)$ denotes a Gamma-variable with shape parameter $\alpha>0$, rate parameter $\lambda>0$, and all such variables are independent. The expectation and variance of $Q_k$ are given by
\begin{align}
\E Q_k = 1,
\qquad
\Var Q_k = \frac{4}{\pi^2}\,\psi_1 \left(\frac{k}{2}\right),
\end{align}
where $\psi_1$ is the trigamma function, defined for $x>0$ by
$
\psi_1(x):=\frac{\dint^2}{\dint x^2}\log\Gamma(x)=\sum_{j=0}^\infty \frac{1}{(j+x)^2}.
$
Finally, as $d\to\infty$ we have
\begin{equation}\label{eq:var_face_count_asymptotics}
\Var \left(\frac{f_{d-q}(\cW_{d+k,d})}{\binom {d+k}{d-q}}\right)
\sim
\frac {\binom{k+q}{2}^2\binom {k+q-2}{k-1}^2} {2^{2k+2q-2}} \frac{\Var Q_k}{d^2} =
\frac{\binom{k+q}{2}^2 \binom {k+q-2}{k-1}^2 \psi_1 \left(\frac k2\right)}{2^{2k+2q-4} \pi^2 \cdot d^2}.
\end{equation}
\end{theorem}

\begin{example}[$k=2$]
For $k=2$, we have $d_{n,2} = 2$ for $n\geq 1$. The distributional representation~\eqref{eq:Q_k_distributional} simplifies to
 \[
Q_2 \; \eqdistr \; \frac{8}{\pi^2}\sum_{n=1,3,5,\dots}\frac{E_n}{n^2},
\]
where $E_1,E_3,E_5,\dots$ are i.i.d.\ exponential with unit mean. It follows that the Laplace transform of $Q_2$ is given by
\[
\mathbb{E}\!\left[\eee^{-sQ_2}\right]
=\prod_{n=1,3,5,\ldots}\mathbb{E}\!\left[\exp\!\left(-s\frac{8}{\pi^2}\frac{E_n}{n^2}\right)\right]
=\prod_{n=1,3,5,\ldots}\frac{1}{1+\frac{8s}{\pi^2 n^2}}
=\frac{1}{\cosh\!\big(\sqrt{2s}\big)},
\qquad s\ge 0.
\]
This distribution is well known and its properties have been reviewed in~\cite{BianePitmanYor2001ZetaTheta}. In particular, the following interpretation is known.
Let $(B_t)_{t\ge 0}$ be standard Brownian motion with $B_0=0$, and define the first exit time from $(-1,1)$ by $\tau_1 \;:=\; \inf\{t\ge 0:\ |B_t|=1\}$. Then $\tau_1$ has the same law as $Q_2$.
\end{example}

Let us record a convex-dual analogue of Theorem~\ref{theo:distributional_limit_face_count_random_cone}.
Let $U_1,\ldots,U_n$ be independent random vectors, each uniformly distributed on $\bS^{d-1}$, where $n>d$.
The hyperplanes $U_1^{\perp},\ldots,U_n^{\perp}$ induce a conical tessellation of $\R^d$ into polyhedral cones.
The number of cones is almost surely constant and equals
$
C(n,d)=2\sum_{\ell=0}^{d-1}\binom{n-1}{\ell}
$,
by the Steiner--Schl\"afli formula; see \cite[Lemma~8.2.1]{schneider_weil_book}.
Choose one of these cones uniformly at random and denote it by $\cS_{n,d}$.
The random cone $\cS_{n,d}$ is called the \emph{random Schl\"afli cone}; see \cite[Definition~3.2]{hug_schneider_conical_tessellations}.

\begin{corollary}[Distributional limit for face counts of Schl\"afli cones]
Fix $k\in\N$ and $q\in\N$. Then, as $d\to\infty$, the number of $q$-dimensional faces of $\cS_{d+k,d}$ satisfies
\begin{equation}
d\cdot\left(\frac{f_q(\cS_{d+k,d})}{\binom{d+k}{\,d-q\,}}-p(k+q,k)\right)
\;{\overset{w}{\underset{d\to\infty}\longrightarrow}}\;
\frac{\binom{k+q}{2}\binom{k+q-2}{k-1}}{2^{k+q-1}}\,(1-Q_k),
\end{equation}
where $Q_k>0$ is the same random variable as in Theorem~\ref{theo:distributional_limit_face_count_random_cone}.
\end{corollary}

\begin{proof}
Let $\cW_{d+k,d}^\circ$ denote the convex dual of $\cW_{d+k,d}$. Then
$f_q(\cW_{d+k,d}^\circ)=f_{d-q}(\cW_{d+k,d})$, and Theorem~\ref{theo:distributional_limit_face_count_random_cone}
can be rewritten as
\begin{equation}\label{eq:f_vect_dual_wendel_cone_large_d_distributional}
d\cdot\left(\frac{f_q(\cW_{d+k,d}^\circ)}{\binom{d+k}{\,d-q\,}}-p(k+q,k)\right)
\;{\overset{w}{\underset{d\to\infty}\longrightarrow}}\;
\frac{\binom{k+q}{2}\binom{k+q-2}{k-1}}{2^{k+q-1}}\,(1-Q_k).
\end{equation}
By~\cite[Theorem~3.1]{hug_schneider_conical_tessellations}, $\cS_{d+k,d}$ has the same distribution as
$\cW_{d+k,d}^\circ$ conditioned on the event $\{\cW_{d+k,d}\neq\R^d\}$.
Moreover, Wendel's formula~\eqref{eq:wendel_formula} yields
\[
\P[\cW_{d+k,d}\neq\R^d]
=1-p(d+k,d)
=1-\frac{1}{2^{d+k-1}}\sum_{\ell=d}^{d+k-1}\binom{d+k-1}{\ell}
\to 1
\qquad \text{as } d\to\infty.
\]
Hence conditioning on $\{\cW_{d+k,d}\neq\R^d\}$ becomes asymptotically negligible, and we may replace
$f_q(\cW_{d+k,d}^\circ)$ by $f_q(\cS_{d+k,d})$ in \eqref{eq:f_vect_dual_wendel_cone_large_d_distributional}
without affecting convergence in distribution.
\end{proof}

\subsection{Background from the theory of \texorpdfstring{$U$}{U}-statistics}\label{subsec:U_stat_information}
The proof of Theorem~\ref{theo:distributional_limit_face_count_random_cone} relies on the theory of \(U\)-statistics. In this section, we introduce the necessary notation and recall several classical results on \(U\)-statistics; see the books of~\citet{Lee1990UStatistics}, \citet{KorolyukBorovskikh1994TheoryUStatistics}, and \citet[Chapter~5]{Serfling1980ApproximationTheorems} for further background.

Let $X_1,X_2,\ldots$ be i.i.d.\ random elements in a measurable space $(\mathcal X,\mathcal A)$
with common law $\mu$. Let $h:\mathcal X^m\to\R$ be a kernel of order $m\in \N$.  Our standing assumption is that $h$ is  measurable, symmetric (invariant under permutations of its arguments), and
$$
\E\bigl[h(X_1,\ldots,X_m)^2\bigr]<\infty.
$$
For $n\ge m$, the \emph{$U$-statistic} associated with $h$ is
\[
\mathbb U_n
:=
\binom{n}{m}^{-1}
\sum_{1\le i_1<\cdots<i_m\le n}
h\!\left(X_{i_1},\ldots,X_{i_m}\right).
\]
Put
\[
g_0 := \E h(X_1,\ldots,X_m) \qquad (\text{so } \E \mathbb U_n = g_0 \text{ for all } n\geq m).
\]
For $c\in \{1,\ldots,m\}$ define the conditional expectation kernels $g_c: \mathcal X^c \to\R$ by
\[
g_c(x_1,\ldots,x_c)
:=
\E\!\left[h(x_1,\ldots,x_c,X_{c+1},\ldots,X_m)\right], \qquad (x_1,\ldots, x_c) \in \mathcal X^c,
\]
where $X_{c+1},\ldots,X_m$ are i.i.d.\ with law $\mu$. Note that $\E g_c (X_1,\ldots, X_c) = g_0$ and define
$$
\widetilde g_c(x_1,\ldots,x_c) = g_c(x_1,\ldots, x_c) - g_0.
$$
Since $h$ is symmetric, each $g_c$ (and $\widetilde g_c$) is symmetric in its arguments.
For $c\in \{1,\ldots,m\}$, define
\[
\zeta_c
:=
\Var g_c(X_1,\ldots,X_c)
=
\E \bigl[\widetilde g_c(X_1,\ldots,X_c)^2\bigr].
\]
For $n\ge m$, the variance of $\mathbb U_n$ is given by the exact formula (see~\cite[Chapter~1, \S\,1.3, Theorem~3]{Lee1990UStatistics} or~\cite[Section~5.2.1, Lemma~A]{Serfling1980ApproximationTheorems})
\begin{equation}\label{eq:U_stat_variance_exact}
\Var \mathbb U_n
=
\sum_{c=1}^{m}
\frac{\binom{m}{c}\binom{n-m}{\,m-c\,}}{\binom{n}{m}}
\;\zeta_c.
\end{equation}
The kernel $h$ (and hence $\mathbb U_n$) is said to be \emph{degenerate of order $c\in\{1,\ldots,m\}$}
if
\[
\zeta_1=\cdots=\zeta_{c-1}=0
\qquad\text{and}\qquad
\zeta_c>0.
\]
Equivalently, $g_r(x_1,\ldots,x_r)=g_0$ for all $r\in \{1,\ldots, c-1\}$ ($\mu^r$-a.e.), but $g_c(X_1,\ldots,X_c)$ is not a.s.\ constant.
If $h$ is degenerate of order $c$, Equation~\eqref{eq:U_stat_variance_exact} implies (see, e.g., \cite[Section~5.3.4]{Serfling1980ApproximationTheorems})
\begin{equation}\label{eq:U_stat_variance_asymptotic}
\Var  \mathbb U_n
=
\frac{\binom{m}{c}\binom{n-m}{m-c}}{\binom{n}{m}}\;\zeta_c
\;+\; o(n^{-c})
\sim
\frac{\binom{m}{c}^2\,c!}{n^c}\;\zeta_c,
\qquad n\to\infty.
\end{equation}

The next classical theorem characterizes the weak limit of the $U$-statistic $\mathbb U_n$ under degeneracy of order $2$; see~\cite[Section~3.2.2]{Lee1990UStatistics}, \cite[Chapter~4]{KorolyukBorovskikh1994TheoryUStatistics}, \cite[p.~168]{RubinVitale1980} or~\cite[Section~5.5.2, p.~194]{Serfling1980ApproximationTheorems}.

\begin{theorem}[Non-Gaussian weak limit under degeneracy of order $2$]\label{theo:weak_limit_U_stat_c_2}
Assume $h:\mathcal X^m\to\R$ is degenerate of order $2$ (equivalently, $\zeta_1 = 0$ and $\zeta_2 >0$). Define the (self-adjoint, Hilbert--Schmidt, hence compact) operator $T:L^2(\mathcal X, \mu)\to L^2(\mathcal X, \mu)$ by
\[
(Tf)(x):=\int_{\mathcal X} \widetilde g_2(x,y)\, f(y)\,\mu(\dint y).
\]
Let $\lambda_1,\lambda_2,\ldots$ be the (real) non-zero eigenvalues of $T$, listed with multiplicities,
and let $\xi_1,\xi_2,\ldots$ be i.i.d.\ standard normal random variables. Then
\[
n\,(\mathbb U_n-\E \mathbb U_n)
\todistr
\binom{m}{2}\,\sum_{j=1}^\infty \lambda_j\,(\xi_j^2-1).
\]
\end{theorem}

The series $\sum_{j=1}^\infty \lambda_j\,(\xi_j^2-1)$ converges in $L^2$ and a.s.\ (since $\sum_{j=1}^\infty \lambda_j^2<\infty$ by the Hilbert--Schmidt property of $T$ and $\E[\xi_j^4] < \infty$) and hence defines a proper random limit.

\subsection{Proof of Theorem~\ref{theo:distributional_limit_face_count_random_cone}: Representation of face count as \texorpdfstring{$U$}{U}-statistic}\label{subsec:proof_theo_f_vect_U_stat_part_0}

Sections~\ref{subsec:proof_theo_f_vect_U_stat_part_0} -- \ref{subsec:proof_theo_f_vect_U_stat_part_4_k_equals_1} are devoted to the proof of Theorem~\ref{theo:distributional_limit_face_count_random_cone}.

The first step is to interpret $f_\ell(\cW_{n,d})/\binom n\ell$ as a $U$-statistic, which becomes possible after passing to the Gale dual. Consider the random cone $\cW_{n,d}$ with $n > d \geq 2$ and let $\ell \in \{0,\ldots, d-1\}$.  With probability $1$, the vectors $U_1,\ldots, U_n$ spanning this cone are in general linear position. Thus, every $\ell$-dimensional face of $\cW_{n,d}$ has the form $\pos (U_i: i\in I)$ for some unique set $I\subseteq [n]$ with $\# I = \ell$ and hence
$$
f_{\ell} (\cW_{n,d}) = \sum_{\substack{I\subseteq [n]\\ \#I = \ell}} \ind_{\{\pos (U_i: i\in  I) \text{ is a face of } \cW_{n,d}\}}.
$$
We now use the coupling given in Theorem~\ref{theo:gale_duality_stoch_orthant} (with $k=n-d$). Recall that the random vectors $V_1,\ldots, V_n$ appearing in that theorem are independent and uniformly distributed on the unit sphere  $\bS^{n-d-1}\subseteq \R^{n-d}$.  By Theorem~\ref{theo:gale_duality_stoch_orthant}, we have the following equality of random events:
\begin{equation*}
\{\pos (U_i: i\in  I) \text{ is a face of } \cW_{n,d}\}  = \{\pos (V_j : j\in I^c)= \R^{n-d}\},
\end{equation*}
for all $I \subseteq [n]$  with $\#I \leq d-1$.
Putting $J:= I^c$ and dividing by $\binom n\ell$ we arrive at the representation
\begin{equation}\label{eq:face_count_gale_dual}
\frac{f_{\ell} (\cW_{n,d})}{\binom n\ell}
=
\binom n{n-\ell}^{-1} \sum_{\substack{J\subseteq [n]\\ \#J = n-\ell}} \ind_{\{\pos (V_j: j\in J) = \R^{n-d}\}}, \qquad \ell \in \{0,\ldots, d-1\}.
\end{equation}
The right-hand side is a $U$-statistic with the sample space $\mathcal X = \bS^{n-d-1}$ (endowed with the uniform distribution $\nu_{n-d}$) and the kernel $h: \mathcal X^{n-\ell}\to \{0,1\}$ given by
$$
h(v_1,\ldots, v_{n-\ell}) = \ind_{\{\pos (v_1,\ldots, v_{n-\ell}) = \R^{n-d}\}},  \qquad v_1,\ldots, v_{n-\ell} \in \bS^{n-d-1}.
$$

Recall that in  Theorem~\ref{theo:distributional_limit_face_count_random_cone} we consider $f_\ell(\cW_{n,d})$, where the number of points is  $n=d+k$, the dimension of faces is $\ell = d-q$, and the parameters  $k\in \N$ and $q \in \N$ are fixed. Let us summarize our setting:
\begin{itemize}
\item The sample space is $\mathcal X = \bS^{k-1}$, endowed with the uniform distribution $\nu_k$.
\item  $V_1,\ldots, V_n$, where $n=d+k$,  are i.i.d.\ random vectors uniformly distributed on $\bS^{k-1}$.
\item The order of the $U$-statistic is $m:= n-\ell =  k+q$ and the kernel $h: \mathcal X^{m} \to \{0,1\}$ is
$$
h(v_1,\ldots, v_m) = \ind_{\{\pos (v_1,\ldots, v_{m}) = \R^{k}\}},  \qquad v_1,\ldots, v_{m} \in \bS^{k-1}.
$$
\end{itemize}
By~\eqref{eq:face_count_gale_dual} we have the identification
\begin{equation}\label{eq:f_vect_as_U_stats_theorem_weak_conv}
\mathbb U_n
=
\binom{d+k}{m}^{-1}
\sum_{1\le i_1<\cdots<i_{m}\le d+k}
\ind_{\{\pos (V_{i_1},\ldots, V_{i_{m}}) = \R^{k}\}}
=
\frac{f_{d-q}(\cW_{d+k,d})}{\binom {d+k}{d-q}}.
\end{equation}

Our aim is to apply Theorem~\ref{theo:weak_limit_U_stat_c_2} to $\mathbb U_n$. In the following Sections~\ref{subsec:proof_theo_f_vect_U_stat_part_1} -- \ref{subsec:proof_theo_f_vect_U_stat_part_3}  we assume that $k\geq 2$ --- the case $k=1$ (in which the sphere becomes degenerate) requires a separate treatment (although many arguments remain valid for $k=1$) and will be considered in Section~\ref{subsec:proof_theo_f_vect_U_stat_part_4_k_equals_1}.

\subsection{Proof of Theorem~\ref{theo:distributional_limit_face_count_random_cone}: Conditional expectation kernels}\label{subsec:proof_theo_f_vect_U_stat_part_1}
In the present section,
we determine the conditional expectation kernels of the $U$-statistic $\mathbb U_n$. In particular, we shall see that $\zeta_1= 0$ and $\zeta_2>0$.

Recall from Section~\ref{subsec:U_stat_information} that for  $c\in \{1,\ldots, m\}$ and unit vectors $v_1,\ldots, v_c\in \bS^{k-1}$,  we consider the conditional expectation kernels
$$
g_c(v_1,\ldots, v_c) = \E h(v_1,\ldots, v_c, V_{c+1},\ldots, V_m) = \P[\pos (v_1,\ldots, v_c, V_{c+1},\ldots, V_m) = \R^k].
$$
Wendel's formula~\eqref{eq:wendel_formula} gives
$$
g_0 = \E h(V_1,\ldots, V_m) =  \P[\pos (V_1,\ldots, V_m) = \R^k] = p(m,k), \qquad \E \mathbb U_n = g_0 = p(m,k).
$$
The centered version of $g_c$ is therefore
$$
\widetilde g_c (v_1,\ldots, v_c) = g_c(v_1,\ldots, v_c) - \E h(V_1,\ldots, V_m) = g_c(v_1,\ldots, v_c) - p(m,k).
$$

\begin{lemma}[$c=1$]\label{eq:u_stat_zeta_1_h_1}
We have $g_1(v_1) = p(m,k)$ and $\widetilde g_1(v_1) = 0$ for every $v_1\in \bS^{k-1}$.  Hence $\zeta_1 = 0$.
\end{lemma}
\begin{proof}
Indeed, by rotation invariance,  for every $v_1\in \bS^{k-1}$ we have
$$
g_1(v_1) = \P[\pos (v_1, V_{2},\ldots, V_m) = \R^k] = \P[\pos (V_1, V_{2},\ldots, V_m) = \R^k] = p(m,k).
$$
Hence $\widetilde g_1(v_1) = g_1(v_1) - p(m,k) = 0$ and $\zeta_1 = \Var\widetilde g_1(V_1) = 0$.
\end{proof}

The next lemma provides an explicit formula for the conditional expectation kernel $g_c(v_1,\ldots, v_c)$ of arbitrary order.
The formula is expressed in terms of conic intrinsic volumes.
To each closed convex cone \(C \subseteq \mathbb{R}^k\) one can associate conic intrinsic volumes \(\upsilon_0(C),\ldots,\upsilon_k(C)\),
which are nonnegative and satisfy \(\upsilon_0(C)+\cdots+\upsilon_k(C)=1\).
We refer to \cite{AmelunxenLotzDCG17,ALMT14} and \cite[Section~2.3]{schneider_book_convex_cones_probab_geom} for their definition and properties.

\begin{lemma}[Conditional expectation kernels] \label{lem:projection_kernels_any_c}
Let $c\in \{1,\ldots, m-1\}$.  Suppose that $v_1,\ldots, v_c \in \bS^{k-1}$, $k\geq 2$,  are such that the polyhedral cone $C= \pos (v_1,\ldots, v_c)\subseteq \R^k$ is not a linear subspace. Then
\begin{equation}\label{eq:lem:projection_kernels_any_c}
g_c(v_1,\ldots, v_c) = 2 \sum_{\substack{i=1,3,5,\ldots \\ i\leq k}} \sum_{j=0}^{k-i} \upsilon_{k-j}(C) \left(\frac{\ind_{\{i+j \neq k\}}}{2^{m-c}}\binom {m-c}{i+j} + \frac{\ind_{\{i+j = k\}}}{2^{m-c}}\sum_{\ell =k}^{m-c} \binom {m-c}{\ell}  \right),
\end{equation}
where $\upsilon_0(C),\ldots, \upsilon_k(C)$ are the conic intrinsic volumes of the cone $C$.
\end{lemma}
\begin{proof}
Put $D:= \pos (V_{c+1},\ldots, V_m)\subseteq\R^k$. Passing to dual cones, we get
$$
g_c(v_1,\ldots, v_c) = \P[\pos (C\cup D) = \R^k] = \P[C^\circ \cap D^\circ = \{0\}].
$$
Let $\mathbb O$ be a random $k\times k$-matrix that is uniformly (Haar) distributed on the group $SO(k)$. Additionally, let $\mathbb O$ be independent of the $\sigma$-algebra $\mathcal V$  generated by $V_1,\ldots, V_n$. Since the law of $D$ (and hence $D^\circ$) is invariant under $SO(k)$, the randomly rotated cone $\mathbb O D^\circ$ has the same distribution as $D^\circ$ and we can write
$$
g_c(v_1,\ldots, v_c) = \P[C^\circ \cap \mathbb O D^\circ = \{0\}] =  \E \Big[\, \P\left[ C^\circ \cap \mathbb O D^\circ = \{0\}\,| \, \mathcal
V\right] \,\Big].
$$

In the conditional probability, the only random element is the random matrix $\mathbb O$. To compute the conditional probability,  we invoke the conic kinematic formula whose proof can be found in~\cite[Corollary~5.2]{AmelunxenLotzDCG17} or in~\cite[Theorem~ 4.3.5]{schneider_book_convex_cones_probab_geom}; for applications see also~\cite{AmelunxenLotzDCG17}.   The formula states that for (deterministic) polyhedral cones $A, B\subseteq \R^k$ that are not both linear subspaces,
$$
\P[A \cap \mathbb O B = \{0\}] = 2 \sum_{\substack{i=1,3,5,\ldots \\ i\leq k}} \upsilon_{k-i}(A \oplus B) = 2 \sum_{\substack{i=1,3,5,\ldots \\ i\leq k}} \sum_{j=0}^{k-i} \upsilon_{j}(A) \upsilon_{k-i-j}(B).
$$
Here, $A\oplus B  := \{(a,b): a\in A, b\in B\} \subseteq \R^k \oplus \R^k$ denotes the direct (orthogonal) sum of $A$ and $B$.
Recall that $v_1,\ldots, v_c \in \bS^{k-1}$ are such that $C= \pos (v_1,\ldots, v_c)$ (and hence $C^\circ$) is not a linear subspace. The kinematic formula gives
$$
\P\left[C^\circ \cap \mathbb O D^\circ = \{0\}\,| \, \mathcal V\right]
=
2 \sum_{\substack{i=1,3,5,\ldots \\ i\leq k}} \sum_{j=0}^{k-i} \upsilon_{j}(C^\circ) \upsilon_{k-i-j}(D^\circ)
=
2 \sum_{\substack{i=1,3,5,\ldots \\ i\leq k}} \sum_{j=0}^{k-i} \upsilon_{k-j}(C) \upsilon_{i+j}(D),
$$
where we used the duality relation $\upsilon_j(C^\circ) = \upsilon_{k-j} (C)$ and similarly for $D$.
Recall that $C$ is deterministic and $D$ is random.  Taking expectation  gives
\begin{equation}\label{eq:proj_kernel_through_expected_conic_tech}
g_c(v_1,\ldots, v_c) = 2 \sum_{\substack{i=1,3,5,\ldots \\ i\leq k}} \sum_{j=0}^{k-i} \upsilon_{k-j}(C) \E \upsilon_{i+j}(D).
\end{equation}
Next recall that $D= \pos (V_{c+1},\ldots, V_m)$ has the same law as the random polyhedral cone $\cW_{m-c, k}$. To proceed, we need formulas for the expected conic intrinsic volumes of this cone.
\begin{proposition}[Expected conic intrinsic volumes of $\cW_{n,d}$]
For all $d\in \N$, $n\in \N$ we have
\begin{align}
&\E \left[\upsilon_k(\cW_{n,d})\right]
=
\E \left[\upsilon_k(\cW_{n,d}) \ind_{\{\cW_{n,d}\neq \R^d\}}\right]
=
\frac 1 {2^n} \binom nk,
\qquad k \in \{0,\dots,d-1\},
\label{eq:donoho_tanner_cone_expect_2}
\\
&\E \left[\upsilon_d(\cW_{n,d})\right]
= 1- \frac 1 {2^n} \sum_{\ell=0}^{d-1}  \binom n\ell = \frac 1 {2^n} \sum_{\ell=d}^{n}  \binom n\ell.
\label{eq:donoho_tanner_cone_expect_3_a}
\end{align}
In the special case where $d\in \N$ and $n\in \{1,\ldots, d\}$, these equations state that
\begin{equation}\label{eq:donoho_tanner_cone_expect_5}
\E [\upsilon_k(\cW_{n,d})] = \frac 1 {2^n} \binom nk,  \qquad   k \in \{0,\dots,n\}.
\end{equation}
\end{proposition}
\begin{proof}
Equations~\eqref{eq:donoho_tanner_cone_expect_2} and~\eqref{eq:donoho_tanner_cone_expect_3_a} are proved in~\cite[Lemma~5.1]{godland_kabluchko_thaele_cones_high_dim_part_I} (building on~\cite[Corollary~4.3]{hug_schneider_conical_tessellations}), where it is assumed that $n\ge d$. See also~\cite[Theorem~5.18]{KabluchkoSteigenberger2025BetaPolytopesCones} for an alternative approach.
We sketch a proof of~\eqref{eq:donoho_tanner_cone_expect_5} for $n\in\{1,\dots,d\}$; for background on conic intrinsic volumes and hyperplane arrangements we refer to~\cite[Section~5.2]{schneider_book_convex_cones_probab_geom}.
For each sign vector $\eps=(\eps_1,\dots,\eps_n)\in\{\pm1\}^n$, consider the random cone
$
C_{\eps}:=\pos(\eps_1U_1,\dots,\eps_nU_n)
$.
By symmetry, $\E\,\upsilon_k(C_\eps)=\E\,\upsilon_k(\cW_{n,d})$ for every $\eps$.
On the other hand, with probability~$1$, the cones $C_\eps$ are precisely the chambers of a hyperplane arrangement in $\lin(U_1,\dots,U_n)$ that has the same intersection poset as the coordinate hyperplane arrangement in $\R^n$, whose chambers are the $2^n$ orthants. Since the $k$th conic intrinsic volume of an orthant in $\R^n$ equals $\frac{1}{2^n}\binom{n}{k}$, the Klivans--Swartz formula~\cite[Theorem~5.1.4]{schneider_book_convex_cones_probab_geom} yields
$
\sum_{\eps\in\{\pm1\}^n}\upsilon_k(C_\eps)=\binom{n}{k}$ a.s.
Taking expectations  gives
$
2^n\,\E\,\upsilon_k(\cW_{n,d})
=\sum_{\eps\in\{\pm1\}^n}\E\,\upsilon_k(C_\eps)
=\binom{n}{k}
$,
which is precisely~\eqref{eq:donoho_tanner_cone_expect_5}.
\end{proof}

Returning to~\eqref{eq:proj_kernel_through_expected_conic_tech} and plugging the values of $\E \upsilon_{i+j}(D) = \E \upsilon_{i+j}(\cW_{m-c, k})$ given by~\eqref{eq:donoho_tanner_cone_expect_2} and~\eqref{eq:donoho_tanner_cone_expect_3_a}  yields
$$
g_c(v_1,\ldots, v_c) = 2 \sum_{\substack{i=1,3,5,\ldots \\ i\leq k}} \sum_{j=0}^{k-i} \upsilon_{k-j}(C) \left(\frac{\ind_{\{i+j \neq k\}}}{2^{m-c}}\binom {m-c}{i+j} + \frac{\ind_{\{i+j = k\}}}{2^{m-c}} \sum_{\ell =k}^{m-c} \binom {m-c}{\ell}  \right).
$$
The proof of Lemma~\ref{lem:projection_kernels_any_c} is complete.
\end{proof}

We have already seen that $g_0 = g_1(v_1) = p(m,k)$ and hence $\widetilde g_1(v_1) = 0$ for all $v_1\in \bS^{k-1}$. The next step is to compute $g_2(v_1,v_2)$.
\begin{lemma}[Conditional expectation kernel for $c=2$]\label{lem:u_stat_proj_kernel_c_2}
Let $v_1,v_2 \in \bS^{k-1}$, $k\geq 2$,  be two vectors such that $v_1\neq \pm v_2$. Then,
$$
\widetilde g_2(v_1,v_2) = \frac 1 {2^{m-1}} \binom {m-2}{k-1} \left(\frac 2 {\pi} \arccos \langle v_1,v_2\rangle -1\right) = - \frac 1 {2^{m-1}} \binom {m-2}{k-1} \cdot \frac 2 {\pi} \arcsin \langle v_1,v_2\rangle.
$$
\end{lemma}
\begin{proof}
First we compute $g_2(v_1,v_2)$ by applying Lemma~\ref{lem:projection_kernels_any_c} with  $c=2$. The cone $C=\pos (v_1,v_2)$ is a wedge with angle $\alpha = \arccos \langle v_1,v_2\rangle$ (measured in radians). Its conic intrinsic volumes are given by $\upsilon_0(C) = \frac 12 - \frac \alpha{2\pi}$, $\upsilon_1(C) = \frac 12$, $\upsilon_2(C) = \frac{\alpha}{2\pi}$, and $\upsilon_m(C) = 0$ whenever $m\geq 3$. So $\upsilon_{k-j}(C) = 0$ unless $k-j\in \{0,1,2\}$. It follows that only the pairs $(i,j) = (1, k-1)$ and $(i,j) = (1, k-2)$ contribute to the sum in~\eqref{eq:lem:projection_kernels_any_c}. So
\begin{align*}
g_2(v_1,v_2)
&=
2 \upsilon_1(C) \cdot \frac{1}{2^{m-2}} \sum_{\ell =k}^{m-2} \binom {m-2}{\ell}  + 2 \upsilon_2(C) \cdot  \frac{1}{2^{m-2}}\binom {m-2}{k-1}
\\
&=
\frac{1}{2^{m-2}} \sum_{\ell =k}^{m-2} \binom {m-2}{\ell}  + \frac{\alpha}{2^{m-2} \pi}\binom {m-2}{k-1}.
\end{align*}
Wendel's formula~\eqref{eq:wendel_formula} and the defining property of the Pascal triangle give
$$
p(m,k) = \P[\cW_{m,k} = \R^k] =  \frac 1  {2^{m-1}} \sum_{\ell=k}^{m-1} \binom{m-1}{\ell} = \frac 1 {2^{m-2}} \sum_{\ell = k}^{m-2} \binom {m-2}{\ell} + \frac 1 {2^{m-1}} \binom {m-2}{k-1}.
$$
Plugging these formulas into $\widetilde g_2(v_1,v_2) = g_2(v_1,v_2) - p(m,k)$ gives the stated formula for $\widetilde g_2(v_1,v_2)$.
\end{proof}

\subsection{Proof of Theorem~\ref{theo:distributional_limit_face_count_random_cone}: The variance}\label{subsec:proof_theo_f_vect_U_stat_part_1_A}
We have already seen that $\zeta_1 = \Var g_1 (V_1) = 0$. In the present section, we compute $\zeta_2 = \Var g_2(V_1,V_2)$. We shall see that $\zeta_2>0$, which means that the $U$-statistic $\mathbb U_n$ is degenerate of order $2$.
\begin{lemma}[Angle between two random unit vectors]\label{lem:angle_between_two_vectors}
Let $V_1,V_2$ be independent and uniformly distributed on $\bS^{k-1}$, $k\geq 2$. Let $\Theta: = \arccos \langle V_1,V_2\rangle \in [0,\pi]$ be the angle between these vectors. Then,
$$
\E \Theta  =  \frac \pi 2,
\qquad
\Var \Theta  =  \frac 12 \psi_1 \left(\frac k2\right),
$$
where $\psi_1(x) = \frac{\dint^2}{\dint x^2}\log\Gamma(x) = \sum_{j=0}^\infty \frac{1}{(j+x)^2}>0$, with  $x>0$, is the trigamma function.
\end{lemma}
\begin{proof}
It is a standard fact, see, e.g., \cite[Example~2.28]{kabluchko_steigenberger_thaele_boob_beta_type}, that $\Theta$ has density
\begin{equation}\label{eq:densitytheta}
f_\Theta(\theta)=\frac{(\sin\theta)^{k-2}}{Z_k},\qquad \theta\in [0,\pi], \qquad
Z_k= \int_0^\pi (\sin u)^{k-2}\,\dint u
= \frac {\sqrt{\pi}\,\Gamma\!\left(\frac{k-1}{2}\right)}{\Gamma\!\left(\frac{k}{2}\right)},
\end{equation}
for $k\geq 2$. It follows that $\Theta$ has the same law as $\pi-\Theta$,  hence $\E \Theta=\pi/2$. Next, it follows that
\begin{equation}\label{eq:Var_Theta_as_quotient_integrals}
\Var(\Theta)=\E\!\left[\left(\Theta-\frac \pi2\right)^2\right]
=\frac{\int_0^\pi \left(\theta-\frac{\pi}{2}\right)^2 (\sin\theta)^{k-2}\,\dint \theta}
{\int_0^\pi (\sin\theta)^{k-2}\,\dint \theta}
=\frac{\int_{-\pi/2}^{\pi/2} x^2 \cos^{k-2}x\,\dint x}
{\int_{-\pi/2}^{\pi/2} \cos^{k-2}x\,\dint x},
\end{equation}
where we used the substitution $x=\theta-\pi/2$. It remains to show that the quotient on the right-hand side equals $\frac12\psi_1(k/2)$.
We use the identity (see~\cite[Eq.~3.631.9]{GradshteynRyzhik2007})
\begin{equation}\label{eq:identity_gradshteyn_ryzhik}
F(a):=\int_{-\pi/2}^{\pi/2}\cos^{k-2}x\,\cos(ax)\,\dint x
=\frac{2\, \pi\,\Gamma(k-1)}{2^{k-1}\,\Gamma\!\big(\frac k2 +\tfrac a2\big)\,
\Gamma\!\big(\frac k2-\tfrac a2\big)},
\qquad k>1, \quad a\in \C.
\end{equation}
On the one hand,  the quotient in \eqref{eq:Var_Theta_as_quotient_integrals} equals $-F''(0)/F(0)$.
On the other hand, taking logarithms in \eqref{eq:identity_gradshteyn_ryzhik}, differentiating twice, and setting $a=0$ gives
$(\log F)'(0)=0$ and $(\log F)''(0)=-\frac12\psi_1(k/2)$. Hence, $-F''(0)/F(0) = -(\log F)''(0) =\frac12\psi_1(k/2)$.
\end{proof}

\begin{corollary}[Formula for $\zeta_2$]\label{cor:u_stat_zeta_2_formula}
For $k\geq 2$ we have
$$
\zeta_2 =  \frac 2 {2^{2m-2} \pi^2}  \binom {m-2}{k-1}^2 \cdot  \psi_1 \left(\frac k2\right) > 0.
$$
\end{corollary}
\begin{proof}
Recall that $\zeta_2=\Var g_2(V_1,V_2) = \Var \widetilde g_2(V_1,V_2)$. Hence
$$
\zeta_2 = \Var \widetilde g_2(V_1,V_2) = \frac 1 {2^{2m-2}} \binom {m-2}{k-1}^2 \cdot \frac 4 {\pi^2} \Var \arccos \langle V_1,V_2\rangle
= \frac 2 {2^{2m-2} \pi^2}  \binom {m-2}{k-1}^2 \cdot  \psi_1 \left(\frac k2\right),
$$
using first Lemma~\ref{lem:u_stat_proj_kernel_c_2} and then Lemma~\ref{lem:angle_between_two_vectors}.
\end{proof}
We are now able to prove the asymptotic formula for the variance stated in~\eqref{eq:var_face_count_asymptotics}.

\begin{proposition}[Asymptotics of the variance]\label{prop:variance_face_count}
Let $k\geq 2$ and $q \in \N$ be fixed. Then
$$
\Var \left(\frac{f_{d-q}(\cW_{d+k,d})}{\binom {d+k}{d-q}}\right) \sim \frac{\binom{k+q}{2}^2 \binom {k+q-2}{k-1}^2 \psi_1 \left(\frac k2\right)}{2^{2k+2q-4} \pi^2 \cdot d^2},
\qquad
d\to\infty.
$$
\end{proposition}

\begin{proof}
We have shown that $\zeta_1= 0$ and $\zeta_2>0$.
Equation~\eqref{eq:U_stat_variance_asymptotic} with $c=2$ gives
$$
\Var \left(\frac{f_{d-q}(\cW_{d+k,d})}{\binom {d+k}{d-q}}\right) = \Var  \mathbb U_n
\sim
\frac{\binom{m}{2}^2\,2!}{n^2}\;\zeta_2
=
\frac{\binom{m}{2}^2 \binom {m-2}{k-1}^2 \psi_1 \left(\frac k2\right)}{2^{2m-4} \pi^2 \cdot n^2}
,
\qquad n\to\infty.
$$
Recalling that $m= k+q$ is fixed  and $n = d+k\sim d$ gives the stated asymptotic equivalence.
\end{proof}

\subsection{Proof of Theorem~\ref{theo:distributional_limit_face_count_random_cone}: Convolution operator on the sphere and its eigenvalues} \label{subsec:proof_theo_f_vect_U_stat_part_2}
Let $k\geq 2$ be an integer and recall that $\nu_k$ is the uniform probability measure on the unit sphere  $\mathbb S^{k-1}=\{v\in\mathbb R^k:\|v\|=1\}$.
Define the integral  operator $A: L^{2}(\mathbb S^{k-1},\nu_k)\to L^{2}(\mathbb S^{k-1},\nu_k)$
by
\begin{equation}\label{eq:A_oper_def}
(Af)(v) =\int_{\mathbb S^{k-1}} \frac{2}{\pi}\arcsin\!\big(\langle v,u\rangle\big) \,f(u)\,\nu_k(\dint u), \qquad f\in L^2(\mathbb S^{k-1},\nu_k), \quad v\in\mathbb S^{k-1}.
\end{equation}
Note that $A$ is Hilbert--Schmidt (hence compact) since its kernel $K(u,v) = \frac{2}{\pi}\arcsin (\langle v,u\rangle)$ is square integrable w.r.t.\ $\nu_k \otimes \nu_k$. Also, $A$ is self-adjoint  since $K$ is real-valued and symmetric in its arguments. By the spectral theorem for compact self-adjoint operators,
the nonzero spectrum of $A$ consists of at most countably
many real eigenvalues of finite multiplicity, with $0$
as the only possible accumulation point.

Our aim is to describe the full spectral decomposition of $A$ (i.e.\ its eigenvalues and the corresponding eigenspaces).  
This will be done in terms of spherical harmonics. Let us recall some basic facts from the harmonic analysis on the sphere; see~\cite[Chapters~1,2]{DaiXu2013}, \cite[Chapter~5]{AxlerEtalHamonicFunctionTheory}, \cite{Muller1966SphericalHarmonics}.
For $n\in \N_0$ let $\mathcal H_n\subseteq L^2 (\bS^{k-1}, \nu_k)$ be the space of spherical harmonics of degree $n$.
One possible definition of $\mathcal H_n$ is this one:
\[
\mathcal H_n=\Big\{Y\in C^\infty(\mathbb S^{k-1}) : \Delta_{\mathbb S^{k-1}}Y=-n(n+k-2)\,Y\Big\},
\]
where $\Delta_{\mathbb S^{k-1}}$ denotes the Laplace--Beltrami operator on $\mathbb S^{k-1}$.  It is known~\cite[Proposition~5.8 on p.~78]{AxlerEtalHamonicFunctionTheory} that the linear space $\mathcal H_n$
is finite-dimensional with
\[
d_{n,k} := \dim \mathcal H_n
= \binom{n+k-1}{n}-\binom{n+k-3}{n-2}.
\]
Moreover,  $\mathcal H_m\perp \mathcal H_n$ for $m\neq n$, and the linear span of $\bigcup_{n\ge 0}\mathcal H_n$ is dense in
$L^2(\mathbb S^{k-1},\nu_k)$; see~\cite[Theorem 2.2.2]{DaiXu2013}. Note that  the excluded case $k=1$ is different since then the Hilbert space $L^2(\mathbb S^{k-1},\nu_k)$ becomes $2$-dimensional --- see Section~\ref{subsec:proof_theo_f_vect_U_stat_part_4_k_equals_1} for this case.

\begin{lemma}[Spectral decomposition of $A$]\label{lem:spectral_decomp_A}
For every $n\in \N_0$ and every degree $n$ spherical harmonic  $Y\in\mathcal H_n$, we have $A Y=\lambda_n \, Y$,  where
\[
\lambda_n = 0 \quad (\text{if $n$ is even}),
\qquad
\qquad
\lambda_n
=
\left(\frac{\Gamma\!\left(\frac{k}{2}\right)\Gamma\!\left(\frac{n}{2}\right)}
{\pi\,\Gamma\!\left(\frac{n+k}{2}\right)}\right)^2 = \frac {\left(B\left(\frac n2, \frac k2\right)\right)^2} {\pi^2}
\quad (\text{if $n$ is odd}).
\]
\end{lemma}

To prove Lemma~\ref{lem:spectral_decomp_A}, we represent $A$ as the square of another operator.
\begin{lemma}[Square root of $A$]
We have $A=S^{2}$, where $S: L^{2}(\mathbb S^{k-1},\nu_k)\to L^{2}(\mathbb S^{k-1},\nu_k)$ is a linear operator defined by
$$
(Sf)(v) = \int_{\mathbb S^{k-1}} \operatorname{sign}\!\big(\langle v,u\rangle\big)\,f(u)\,\nu_k(\dint u),
\qquad f\in L^2(\mathbb S^{k-1},\nu_k), \quad v\in\mathbb S^{k-1}.
$$
\end{lemma}
\begin{proof}
Let $U$ be uniformly distributed on $\bS^{k-1}$. For all $v,w\in \bS^{k-1}$ it is easy to check  (see, e.g., \cite[Lemma 6.7 on p. 144]{WilliamsonShmoys2011Design}) that
$$
\E\!\left[\operatorname{sign}\!\big(\langle v,U\rangle\big)\operatorname{sign}\!\big(\langle U,w\rangle\big)\right]
=\frac{2}{\pi}\arcsin(\langle v,w\rangle).
$$
Applying the definition of $S$ twice gives
\[
(S^2 f)(v)
=\int_{\mathbb S^{k-1}}\!\!\left(\int_{\mathbb S^{k-1}}
\operatorname{sign}\!\big(\langle v,u\rangle\big)\operatorname{sign}\!\big(\langle u,w\rangle\big)\,\nu_k(\dint u)\right)
f(w)\,\nu_k(\dint w).
\]
Thus $S^2$ is an integral operator with kernel
\[
L(v,w):=\int_{\mathbb S^{k-1}}
\operatorname{sign}\!\big(\langle v,u\rangle\big)\operatorname{sign}\!\big(\langle u,w\rangle\big)\,\nu_k(\dint u) = \E\!\left[\operatorname{sign}\!\big(\langle v,U\rangle\big)\operatorname{sign}\!\big(\langle U,w\rangle\big)\right] = \frac{2}{\pi}\arcsin(\langle v,w\rangle).
\]
We conclude that  $A=S^2$.
\end{proof}

The next lemma describes the spectral decomposition of $S$ and, together with the formula $A=S^2$, implies Lemma~\ref{lem:spectral_decomp_A}.

\begin{lemma}[Spectral decomposition of $S$]\label{lem:spectral_decomp_S}
For every $n\in \N_0$ and every $Y\in\mathcal H_n$, we have $SY=\tau_n\,Y$  (i.e.\ the restriction of $S$ to $\mathcal H_n$ is the multiplication by $\tau_n$), where
\[
\tau_n = 0 \quad (\text{if $n$ is even}),
\qquad
\qquad
\tau_n
=
(-1)^{\frac{n-1}{2}}\,
\frac{\Gamma\!\left(\frac{k}{2}\right)\Gamma\!\left(\frac{n}{2}\right)}
{\pi\,\Gamma\!\left(\frac{n+k}{2}\right)},
\quad (\text{if $n$ is odd}).
\]
\end{lemma}
\begin{proof}
For $n=0$, we have $S1=0$. Hence,
we may assume that $n\geq 1$.

We first consider the case $k=2$. Parametrize
$\mathbb S^1$ by
$v(\theta)=(\cos\theta,\sin\theta)$,
$\theta\in[0,2\pi)$. Then
\[
(Sf)(v(\theta))
=\frac{1}{2\pi}\int_0^{2\pi}
\operatorname{sign}(\cos(\theta-\varphi))
f(v(\varphi))\,\dint\varphi.
\]
Recall that $\mathcal H_n$ is spanned by
$\cos(n\theta)$ and $\sin(n\theta)$ for $n\geq 1$.
Since the function
$K(t):=\operatorname{sign}(\cos t)$
is even and $2\pi$-periodic, we have $\int_0^{2\pi} K(t)\sin(nt)\,\dint t=0$.
Hence, a change of variables and the addition
formulas for sine and cosine show that both
$\cos(n\theta)$ and $\sin(n\theta)$ are
eigenfunctions of $S$ with the same eigenvalue
\[
\tau_n
=\frac{1}{2\pi}\int_0^{2\pi}
\operatorname{sign}(\cos t)\cos(nt)\,\dint t
=\frac{2}{\pi n}\sin\left(\frac{n\pi}{2}\right)
=
\begin{cases}
0, & n\text{ even},\\[2mm]
(-1)^{\frac{n-1}{2}}\frac{2}{\pi n},
& n\text{ odd}.
\end{cases}
\]
For odd $n$, the identity
$\Gamma(\frac{n+2}{2})
=\frac n2\Gamma(\frac n2)$ shows that
this agrees with the formula stated in the lemma.

In the remainder of the proof, we assume that
$k\geq 3$. Our main tool is the Funk--Hecke
formula~\cite[Theorem~1.2.9]{DaiXu2013},
which we now recall.
Let $\alpha=\frac{k-2}{2}>0$ and let
$h:[-1,1]\to\mathbb C$ be integrable with
respect to the measure
$(1-t^2)^{\alpha-\frac12}\,\dint t$.
Define the integral operator $T_h: L^{2}(\mathbb S^{k-1},\nu_k)\to L^{2}(\mathbb S^{k-1},\nu_k)$ by
\[
(T_h f)(v):=\int_{\mathbb S^{k-1}} h(\langle v,u\rangle)\,f(u)\, \nu_k(\dint u).
\]
The Funk--Hecke formula states that  for every  $n\in \N_0$ and every degree $n$ spherical harmonic $Y\in\mathcal H_n$, we have $T_h Y=\lambda_n(h)\,Y$, where
\[
\lambda_n(h)
=\frac{b_\alpha}{C_n^{(\alpha)}(1)}
\int_{-1}^1 h(t)\,C_n^{(\alpha)}(t)\,(1-t^2)^{\alpha-\frac12}\,\dint t,
\qquad
b_\alpha=\left(\int_{-1}^1 (1-t^2)^{\alpha-\frac12}\,\dint t\right)^{-1}
=\frac{\Gamma(\alpha+1)}{\sqrt{\pi}\,\Gamma(\alpha+\tfrac12)}.
\]
%%%% I checked the formula and the constants. Note that \beta_\alpha = \omega_{k-1}/omega_k - in the book the integral is taken over a non-normalized measure on the sphere.
Here $C_n^{(\alpha)}$ denotes the Gegenbauer polynomial~\cite[Sections~B.1,B.2]{DaiXu2013} of degree $n$. For our purposes, it is convenient to define it by the Rodrigues' formula (see~\cite[Equations~(B.1.2), (B.2.1)]{DaiXu2013})
\[
C_n^{(\alpha)}(t)\, (1-t^2)^{\alpha-\frac12}
=\kappa_n^{(\alpha)}\,\frac{\dint^n}{\dint t^n}\Big((1-t^2)^{n+\alpha-\frac12}\Big),
\qquad
\kappa_n^{(\alpha)}=\frac{(-1)^n}{2^n n!}\,
\frac{\Gamma\!\left(\alpha+\tfrac12\right)\Gamma(n+2\alpha)}
{\Gamma(2\alpha)\Gamma\!\left(n+\alpha+\tfrac12\right)}.
\]
%ZK: I checked this formula.
In our setting, $h(t)=\operatorname{sign}(t)$, and the Funk--Hecke formula gives $SY=\tau_n\,Y$ for all $Y\in \mathcal H_n$, where
\[
\tau_n=\frac{b_\alpha}{C_n^{(\alpha)}(1)}\, I_n^{(\alpha)}
\qquad
\text{ with }
\qquad
I_n^{(\alpha)}:= \int_{-1}^{1} \operatorname{sign}(t)\,C_n^{(\alpha)}(t)\,(1-t^2)^{\alpha-\frac12}\,\dint t.
\]
It remains to compute $I_n^{(\alpha)}$.
Writing $F_n(t):=(1-t^2)^{n+\alpha-\frac12}$, splitting the integral and using the fact that the antiderivative of $F_n^{(n)}$ is $F_n^{(n-1)}$ gives
\[
I_n^{(\alpha)}=\kappa_n^{(\alpha)} \int_{-1}^1 \operatorname{sign}(t)\,F_n^{(n)}(t)\,\dint t
=
- \kappa_n^{(\alpha)} F_n^{(n-1)}(t) \big |_{t=-1}^{t=0} + \kappa_n^{(\alpha)} F_n^{(n-1)}(t) \big |_{t=0}^{t=1}.
\]
Since the vanishing order of $F_n(t)$ at $t= \pm 1$ is $n+\alpha-\frac12 > n-1$, we have $F_n^{(n-1)}(\pm 1) = 0$.  It follows that
$$
I_n^{(\alpha)} = -2\kappa_n^{(\alpha)}\,F_n^{(n-1)}(0).
$$
Taking these results together gives
\begin{equation}\label{eq:tau_n_tech_123}
\tau_n=\frac{b_\alpha}{C_n^{(\alpha)}(1)}\,I_n^{(\alpha)} =  \frac{b_\alpha}{C_n^{(\alpha)}(1)}\, \cdot (-2\kappa_n^{(\alpha)})  F_n^{(n-1)}(0) = \frac{(-1)^{n+1}}{2^{n-1}}\,\frac{\Gamma(\alpha+1)}{\sqrt \pi}\,
\frac{F_n^{(n-1)}(0)}{\Gamma(n+\alpha + \frac 12)},
\end{equation}
where we used that
\[
\kappa_n^{(\alpha)}=\frac{(-1)^n}{2^n n!}\,
\frac{\Gamma\!\left(\alpha+\frac12\right)\Gamma(n+2\alpha)}
{\Gamma(2\alpha)\Gamma\!\left(n+\alpha+\frac12\right)},\qquad
b_\alpha=\frac{\Gamma(\alpha+1)}{\sqrt\pi\,\Gamma\!\left(\alpha+\frac12\right)},\qquad
C_n^{(\alpha)}(1)=\frac{\Gamma(n+2\alpha)}{\Gamma(2\alpha)\,n!};
\]
see~\cite[Equation~(B.2.2)]{DaiXu2013} for the formula for $C_n^{(\alpha)}(1)$.
The function $F_n(t)$ is even and~\eqref{eq:tau_n_tech_123} implies that  $\tau_n=0$ for even $n$.  Let $n$ be odd.  Expanding around $t=0$ gives
\[
F_n(t)=(1-t^2)^{\,n+\alpha-\frac12} = \sum_{j=0}^\infty (-1)^j\binom{n+\alpha-\frac12}{j}\,t^{2j}.
\]
Then $n-1$ is even and the term contributing to $t^{n-1}$ is $j=\frac{n-1}{2}$, hence
\[
F_n^{(n-1)}(0)
=(n-1)!\,(-1)^{\frac{n-1}{2}}
\binom{n+\alpha-\frac12}{\frac{n-1}{2}}
=(-1)^{\frac{n-1}{2}}\,
\frac{\Gamma(n) \, \Gamma\!\left(n+\alpha+\frac12\right)}
{\Gamma\!\left(\frac{n+1}{2}\right)\,
 \Gamma\!\left(\alpha+\frac{n+2}{2}\right)}.
\]
Plugging this into~\eqref{eq:tau_n_tech_123} and using Legendre's duplication formula $\Gamma(n)
=2^{\,n-1}\,\frac{\Gamma\!\left(\frac{n}{2}\right)\,\Gamma\!\left(\frac{n+1}{2}\right)}{\sqrt{\pi}}$ gives
\[
\tau_n
=
(-1)^{\frac{n-1}{2}}\,
\frac{\Gamma(\alpha+1)\,\Gamma\!\left(\frac{n}{2}\right)}
{\pi\,\Gamma\!\left(\alpha+\frac{n}{2}+1\right)}
=
(-1)^{\frac{n-1}{2}}\,
\frac{\Gamma\!\left(\frac{k}{2}\right)\Gamma\!\left(\frac{n}{2}\right)}
{\pi\,\Gamma\!\left(\frac{n+k}{2}\right)},
\qquad n\ \text{odd}.
\]
In the last step, we used  $\alpha=\frac{k-2}{2}$.
\end{proof}

\begin{proposition}[Traces of $A$ and $A^2$]\label{prop:traces_A_A_square}
The operator $A$ defined in~\eqref{eq:A_oper_def} is of trace class and
\begin{align*}
\operatorname{tr}(A)
&=
\sum_{n = 1,3,5,\ldots} d_{n,k}\, \left(\frac{\Gamma\!\left(\frac{k}{2}\right)\Gamma\!\left(\frac{n}{2}\right)}
{\pi\,\Gamma\!\left(\frac{n+k}{2}\right)}\right)^2 \;=\; 1,
\\
\operatorname{tr}(A^2)
&=
\sum_{n = 1,3,5,\ldots} d_{n,k}\, \left(\frac{\Gamma\!\left(\frac{k}{2}\right)\Gamma\!\left(\frac{n}{2}\right)}
{\pi\,\Gamma\!\left(\frac{n+k}{2}\right)}\right)^4
=\frac{2}{\pi^2}\,\psi_1\!\left(\frac{k}{2}\right),
\end{align*}
where
$
d_{n,k} = \dim \mathcal H_n=\binom{n+k-1}{n}-\binom{n+k-3}{n-2}
$
and $\psi_1$ is the trigamma function.
\end{proposition}
\begin{proof}
Let us first show that $A, A^2$ are trace class operators and  $\operatorname{tr}(A)= 1$,  $\operatorname{tr}(A^2)= \frac 2{\pi^2} \psi_1 (\frac k2)$.
Recall that $A=S^2= S S^*$ and $S=S^*$ is Hilbert--Schmidt. It follows that $A$ is trace class and $\operatorname{tr}(A)= \|S\|^2_{\text{HS}}$ is the squared Hilbert--Schmidt norm of $S$. Since the kernel of $S$ takes only values $\pm 1$, we have $\|S\|_{\text{HS}}^2 = 1$ and hence $\operatorname{tr}(A) = 1$. Similarly, $A^2 = A A^*$ and $A=A^*$ is Hilbert--Schmidt, hence $A^2$ is of trace class  and
\begin{align*}
\operatorname{tr}(A^2)
&= \|A\|_{\text{HS}}^2
=
\frac{4}{\pi^2} \iint_{\mathbb S^{k-1}\times \mathbb S^{k-1}} \arcsin^2\!\big(\langle v,u\rangle\big) \,\nu_k(\dint v)\,\nu_k(\dint u)
=
\frac{4}{\pi^2} \int_{\bS^{k-1}} \arcsin^2\!\big(\langle v,e_1\rangle\big) \,\nu_k(\dint v)
\\
&=
\frac{4}{\pi^2}\,\frac{\Gamma\!\left(\frac{k}{2}\right)}{\sqrt{\pi}\,\Gamma\!\left(\frac{k-1}{2}\right)} \int_{-1}^1 \arcsin^2(t)\,(1-t^2)^{\frac{k-3}{2}}\,\dint t
=\frac{4}{\pi^2}\, \frac{\Gamma\!\left(\frac{k}{2}\right)}{\sqrt{\pi}\,\Gamma\!\left(\frac{k-1}{2}\right)} \int_{-\pi/2}^{\pi/2} x^2\,\cos^{k-2}x\,\dint x,
\end{align*}
where the last step follows by the substitution $t= \sin x$. We have already computed the right-hand side in the proof of Lemma~\ref{lem:angle_between_two_vectors}: it equals $\frac 4{\pi^2} \Var \Theta  =  \frac 2{\pi^2} \psi_1 (\frac k2)$. So, $\operatorname{tr}(A^2)= \frac 2{\pi^2} \psi_1 (\frac k2)$.

On the other hand, we have seen in Lemma~\ref{lem:spectral_decomp_A} that the nonzero eigenvalues of $A$ are
\[
\lambda_n=\left(\frac{\Gamma\!\left(\frac{k}{2}\right)\Gamma\!\left(\frac{n}{2}\right)}
{\pi\,\Gamma\!\left(\frac{n+k}{2}\right)}\right)^2 \qquad (\text{with multiplicity } d_{n,k} = \dim \mathcal H_n),   \quad n=1,3,5,\ldots.
\]
Note that $0$ is also an eigenvalue of $A$, and the eigenspace of $0$ is the closure of $\oplus_{n=0,2,4,\ldots} \mathcal H_n$; however, $0$ does not contribute to the trace.  It follows that $\operatorname{tr}(A) = \sum_{n=1,3,5,\ldots} \lambda_n d_{n,k}$ and $\operatorname{tr}(A^2) = \sum_{n=1,3,5,\ldots} \lambda_n^2 d_{n,k}$, which completes the proof.
\end{proof}

\begin{example}[$k=2$]
For $k=2$ the underlying space is the unit circle $\mathbb S^{1}$. We represent $v\in \bS^1$ as $v= (\cos \theta, \sin \theta)$. The spaces of spherical harmonics are as follows:  $\mathcal H_0=\lin\{1\}$ and $\mathcal H_n= \lin \{\cos (n\theta), \sin (n\theta)\}$  for $n\geq 1$.    So $\dim  \mathcal H_0=1$ and $\dim \mathcal H_n=2$ for $n\ge 1$.
The non-zero eigenvalues of the operator $A$ are $4/(\pi^2 n^2)$, $n=1,3,5,\dots$, each with multiplicity $2$. This can be verified by direct computation.
\end{example}

\subsection{Proof of Theorem~\ref{theo:distributional_limit_face_count_random_cone}: Completing the argument} \label{subsec:proof_theo_f_vect_U_stat_part_3}
We have now gathered all ingredients needed to prove Theorem~\ref{theo:distributional_limit_face_count_random_cone} for $k\geq 2$.  Recall from~\eqref{eq:f_vect_as_U_stats_theorem_weak_conv} the representation
$$
f_{d-q}(\cW_{d+k,d})/\binom {d+k}{d-q} = \mathbb U_n,
$$
where $\mathbb U_n$ is the $U$-statistic  corresponding to the kernel $h(v_1,\ldots, v_m) = \ind_{\{\pos (v_1,\ldots, v_{m}) = \R^{k}\}}$, for all  $v_1,\ldots, v_{m} \in \bS^{k-1}$. In Lemma~\ref{eq:u_stat_zeta_1_h_1} we have shown that $\zeta_1 = 0$ and in Lemma~\ref{lem:u_stat_proj_kernel_c_2} we computed
$$
\widetilde g_2(v_1,v_2)
=
- \frac 1 {2^{m-1}} \binom {m-2}{k-1} \cdot \frac 2 {\pi} \arcsin \langle v_1,v_2\rangle,
\qquad
v_1\neq \pm v_2.
$$
In Corollary~\ref{cor:u_stat_zeta_2_formula} we have seen that $\zeta_2 >0$. We are thus in the setting of Theorem~\ref{theo:weak_limit_U_stat_c_2}.
The operator $T:L^2(\bS^{k-1}, \nu_k)\to L^2(\bS^{k-1}, \nu_k)$ appearing in  Theorem~\ref{theo:weak_limit_U_stat_c_2} is
$$
(Tf)(x):=\int_{\bS^{k-1}} \widetilde g_2(x,y)\, f(y)\,\nu_k(\dint y) =  - \frac 1 {2^{m-1}} \binom {m-2}{k-1} \cdot (Af)(x),
$$
where $A$ is the operator we analyzed in Section~\ref{subsec:proof_theo_f_vect_U_stat_part_2}. We have seen in Lemma~\ref{lem:spectral_decomp_A} that the non-zero eigenvalues of the operator $A$ are
$$
\lambda_r = \left(\frac{\Gamma\!\left(\frac{k}{2}\right)\Gamma\!\left(\frac{r}{2}\right)}
{\pi\,\Gamma\!\left(\frac{r+k}{2}\right)}\right)^2
\text{ with multiplicity }
d_{r,k}
= \binom{r+k-1}{r}-\binom{r+k-3}{r-2},
\qquad
r=1,3,5,\ldots.
$$
The non-zero eigenvalues of $T$ are then $- \frac 1 {2^{m-1}} \binom {m-2}{k-1} \, \lambda_r$ with  $r=1,3,5,\ldots$ and the same multiplicities $d_{r,k}$ as before.  Now we have everything to apply Theorem~\ref{theo:weak_limit_U_stat_c_2}.
Theorem~\ref{theo:weak_limit_U_stat_c_2} gives
\[
n\,(\mathbb U_n-\E \mathbb U_n)
\todistr
-\frac {\binom{m}{2}\binom {m-2}{k-1}} {2^{m-1}} \,\sum_{r=1,3,5,\ldots} \sum_{j=1}^{d_{r,k}} \lambda_r \,(\xi_{r,j}^2-1),
\]
where $\xi_{r,j}$ are i.i.d.\ standard normal random variables indexed by $r\in \{1,3,5,\ldots\}$ and  $j\in \{1,\ldots, d_{r,k}\}$. In Proposition~\ref{prop:traces_A_A_square} we have seen that $\operatorname{tr}(A)  = \sum_{r=1,3,5,\ldots}^\infty d_{r,k} \lambda_r = 1$, which allows us to rewrite the above result as
\[
n\,(\mathbb U_n-\E \mathbb U_n)
\todistr
\frac {\binom{m}{2}\binom {m-2}{k-1}} {2^{m-1}} \, (1-Q_k) \qquad \text{ with } \qquad Q_k:= \sum_{r=1,3,5,\ldots} \lambda_r \, \sum_{j=1}^{d_{r,k}} \xi_{r,j}^2.
\]
Now observe that $\sum_{j=1}^{d_{r,k}} \xi_{r,j}^2$ is distributed as $\operatorname{Gamma} (\frac{d_{r,k}}{2}, \frac 12)$  (chi squared distribution) and these variables are independent for different $r$'s. This proves the distributional representation for $Q_k$ stated in Theorem~\ref{theo:distributional_limit_face_count_random_cone}. By Proposition~\ref{prop:traces_A_A_square},  $\E Q_k = \operatorname{tr}(A)= 1$. Similarly, the variance of $Q_k$ is given by
$$
\Var Q_k = \sum_{r=1,3,5,\ldots} \lambda_r^2 \, \sum_{j=1}^{d_{r,k}} \Var (\xi_{r,j}^2) =  \operatorname{tr}(A^2)  \Var (\xi_{1,1}^2) = \frac{2}{\pi^2}\,\psi_1\!\left(\frac{k}{2}\right) \Var (\xi_{1,1}^2) = \frac{4}{\pi^2}\,\psi_1\!\left(\frac{k}{2}\right),
$$
where the formula for $\operatorname{tr}(A^2)$ comes from Proposition~\ref{prop:traces_A_A_square}. The asymptotic formula for $\Var \mathbb U_n$ stated in~\eqref{eq:var_face_count_asymptotics} has been already established in Proposition~\ref{prop:variance_face_count}.
The proof of Theorem~\ref{theo:distributional_limit_face_count_random_cone} in the case where $k\geq 2$ is complete.

\subsection{Proof of Theorem~\ref{theo:distributional_limit_face_count_random_cone}: Case~\texorpdfstring{$k=1$}{k=1}} \label{subsec:proof_theo_f_vect_U_stat_part_4_k_equals_1}
The case where $k=1$ requires special treatment since the unit sphere becomes degenerate,  $\mathbb S^{0}=\{\pm1\}$, and the parts of the proof based on the conic kinematic formula and spherical harmonics do not apply directly. Our setting is as follows:

\begin{itemize}
\item The sample space is $\mathcal X = \mathbb S^{0} = \{\pm1\}$, endowed with the uniform distribution $\nu_1$.
\item  $V_1,V_2,\ldots,V_n$, where $n= d+1$,  are i.i.d.\ with
$\P[V_i=1]=\P[V_i=-1]=1/2$.
\item The order of the $U$-statistic is $m:= n-\ell = q+1$ and the kernel $h: \{\pm 1\}^{m} \to \{0,1\}$ is
\[
h(v_1,\ldots,v_m)
:=\ind_{\{\pos(v_1,\ldots,v_m)=\R\}}
=
1-\ind_{\{v_1=\ldots= v_m\}}.
\]
\end{itemize}
The associated $U$-statistic is $\mathbb U_n = f_{d-q}(\cW_{d+1,d})/\binom {d+1}{d-q}$; see~\eqref{eq:f_vect_as_U_stats_theorem_weak_conv}.
\begin{lemma}[Conditional expectation kernels for $k=1$]
We have $g_0=\E \mathbb U_n = 1-2^{1-m}$, $\widetilde g_1(x) = 0$ for $x\in\{\pm1\}$ and
\begin{equation}\label{eq:h_2_centered_k_equals_1}
\widetilde g_2(x,y)=
\begin{cases}
+2^{1-m}, & x\neq y,\\
-2^{1-m}, & x=y,
\end{cases}
\qquad x,y\in\{\pm1\}.
\end{equation}
Consequently, $\zeta_1= 0$ and $\zeta_2 =2^{2-2m}$.
\end{lemma}
\begin{proof}
Clearly,
\[
g_0 =\E \mathbb U_n=\E h(V_1,\ldots,V_m)
=1-\P[V_1=\ldots=V_m] = 1-2\cdot 2^{-m}.
\]
Next, for $x\in\{\pm1\}$,
\[
g_1(x) = \E\big[h(x,V_2,\ldots,V_m)\big]
=1-\P\big[x,V_2,\ldots,V_m\text{ are all equal}\big] = 1-2^{-(m-1)} = g_0.
\]
Hence $\widetilde g_1(x) = g_1(x) -g_0 = 0$ and $\zeta_1 = 0$.  Next, for $x,y\in\{\pm1\}$,
\[
g_2(x,y)=\E\big[h(x,y,V_3,\ldots,V_m)\big]
=1-\P\big[x,y,V_3,\ldots,V_m\text{ are all equal}\big].
\]
If $x\neq y$, then the event ``all equal''
is impossible and hence $g_2(x,y)=1$.
If $x=y$, then ``all equal'' occurs iff $V_3=\cdots=V_m=x$, which has probability
$2^{-(m-2)}$. Therefore
\[
g_2(x,y)=
\begin{cases}
1, & x\neq y,\\[1mm]
1-2^{-(m-2)}, & x=y,
\end{cases}
\qquad x,y\in\{\pm1\}.
\]
Hence the centered kernel $\widetilde g_2(x,y)=g_2(x,y)-g_0 =g_2(x,y)-  1 + 2\cdot 2^{-m} $ is given by~\eqref{eq:h_2_centered_k_equals_1}.
Finally,
$
\zeta_2
=\E[\widetilde g_2(V_1,V_2)^2]
=(2^{1-m})^2
=2^{2-2m}.
$
\end{proof}

As a linear space, $L^2(\{\pm 1\},\nu_1)$ can be identified with $\R^2$ by identifying a function $f$ with the vector $(f(+1),f(-1))^\top$. The integral operator $T:L^2(\{\pm 1\},\nu_1)\to L^2(\{\pm 1\},\nu_1)$ with kernel $\widetilde g_2(x,y)$ appearing in Theorem~\ref{theo:weak_limit_U_stat_c_2} is then represented by the matrix
\[
T \;\equiv\; 2^{-m}
\begin{pmatrix}
-1 & \phantom{-}1\\
\phantom{-}1 & -1
\end{pmatrix}.
\]
In particular, $T$ has eigenvalues $0$ and $-2^{1-m}$. Recall that $n=d+1\sim d$ and $m=q+1$. Theorem~\ref{theo:weak_limit_U_stat_c_2} yields
\begin{align*}
n\cdot \left(\frac{f_{d-q}(\cW_{d+1,d}) - \E f_{d-q}(\cW_{d+1,d})}{\binom {d+1}{d-q}}\right)
=
n\,(\mathbb U_n-\E \mathbb U_n)
\ & \todistr\
-\binom{m}{2}\,2^{1-m}\,(\xi^2-1)
\\
&=
\binom{q+1}{2}\,2^{-q}\,(1-\xi^2),
\end{align*}
where $\xi$ is standard normal and hence $Q_1:= \xi^2 \eqdistr  \operatorname{Gamma}(\frac 12, \frac 12)$.
Applying~\eqref{eq:U_stat_variance_asymptotic} and recalling that $\zeta_2 = 2^{-2q}$ gives
$$
\Var \left(\frac{f_{d-q}(\cW_{d+1,d})}{\binom {d+1}{d-q}}\right)
=
\Var  \mathbb U_n
\sim
\frac{2\, \binom{m}{2}^2}{n^2}\;\zeta_2
\sim
\frac{\binom{q+1}{2}^2}{d^2}\;2^{1-2q}.
$$
Note finally that $\E [\xi^2] = 1$  and $\Var [\xi^2] = 2 = \frac{4}{\pi^2}\,\psi_1 (\frac{1}{2})$. This proves Theorem~\ref{theo:distributional_limit_face_count_random_cone} for $k=1$.

\subsection{Exponential concentration of face counts}
Let us finally  mention the following non-asymptotic concentration result. Recall that $p(n,k)$ is given by Wendel's formula~\eqref{eq:wendel_formula}.
\begin{theorem}[Exponential concentration of face counts]\label{theo:exp_concentration_face_count}
Let $d\in \N$, $n > d$, and $\ell \in \{0,\ldots, d-1\}$.  Then, for all $t>0$,
$$
\P\left[\left|\frac{f_\ell(\cW_{n,d})}{\binom n\ell} - p(n-\ell, n-d)\right| \geq t\right] \leq 2 \eee^{-2t^2 \lfloor n/(n-\ell) \rfloor}.
$$
\end{theorem}

For the proof we need a result of~\citet{Hoeffding1963ProbabilityInequalities} which  can be found in the book of~\citet[Theorem~A on p.~201]{Serfling1980ApproximationTheorems} and in~\cite{Pitcan2019NoteConcentrationUStatistics}.

\begin{lemma}[Hoeffding bound for $U$-statistics]\label{lem:hoeffding_for_U_stat}
Let the kernel $h:\mathcal{X}^m\to\R$ be bounded and
nonconstant, and put
\[
L:=\sup_{x\in\mathcal{X}^m}h(x)
   -\inf_{x\in\mathcal{X}^m}h(x)\in(0,\infty).
\]
Then for every $n \geq m$ and $t>0$, we have
$
\P \left[\left|\mathbb U_n-\E \mathbb U_n\right|\ge t \right]
\;\le\;
2\eee^{- 2\lfloor n/m\rfloor\,t^2 / L^2}.
$
\end{lemma}

\begin{proof}[Proof of Theorem~\ref{theo:exp_concentration_face_count}]
Recall from Section~\ref{subsec:proof_theo_f_vect_U_stat_part_0} (in particular,  Equation~\eqref{eq:face_count_gale_dual}) the identification
$$
\mathbb U_n
=
\binom{n}{n-\ell}^{-1}
\sum_{1\le i_1<\cdots<i_{n-\ell}\le n}
\ind_{\{\pos (V_{i_1},\ldots, V_{i_{n-\ell}}) = \R^{n-d}\}}
=
\frac{f_\ell(\cW_{n,d})}{\binom n\ell},
$$
where $X_i:= V_i$, $i\in \N$, are i.i.d.\ random elements with the uniform distribution on $\mathcal X:= \bS^{n-d-1}$.
By Proposition~\ref{prop:expected_face_number_donoho_tanner_cone}, we have $\E \mathbb U_n = p(n-\ell, n-d)$.  Lemma~\ref{lem:hoeffding_for_U_stat} gives $$
\P\left[\left|\frac{f_\ell(\cW_{n,d})}{\binom n\ell} - p(n-\ell, n-d)\right| \geq t\right]
=
\P \left[\left|\mathbb U_n-\E \mathbb U_n\right|\ge t \right]
\;\le\;
2\eee^{- 2\lfloor n/m\rfloor\,t^2 / L^2}
=
2 \eee^{-2t^2 \lfloor n/(n-\ell) \rfloor},
$$
since $h$ takes values $0,1$ only and hence $L= \sup_{x\in\mathcal{X}^m} h(x) - \inf_{x\in\mathcal{X}^m} h(x) = 1$.
\end{proof}

\appendix
\section{Basic facts on Gale duality}\label{sec:gale_duality_basic_facts}
In this appendix we collect a number of standard facts about (linear) Gale duality. Although
these results are well known, we include proofs for completeness.

\medskip

We first fix notation and briefly recall linear Gale duality. Let $d',d''\in\mathbb N$ and set $n:=d'+d''$. Consider vector configurations
\[
a_1,\dots,a_n\in\mathbb R^{d'}
\qquad\text{and}\qquad
b_1,\dots,b_n\in\mathbb R^{d''}.
\]
Let $A\in\mathbb R^{d'\times n}$ be the matrix with columns $a_1,\dots,a_n$ and
$B\in\mathbb R^{d''\times n}$ the matrix with columns $b_1,\dots,b_n$.
The row space of $A$, denoted by $\Row(A)$, is the linear subspace of $\mathbb R^n$
spanned by the rows of $A$; the row space $\Row(B)$ is defined analogously.
We recall from Section~\ref{subsec:gale_transforms} that the configurations
$a_1,\dots,a_n$ and $b_1,\dots,b_n$ are in \emph{linear Gale duality} if
\begin{equation}\label{eq:Gale_duality_row_spaces_orthogonal}
\dim\Row(A)=d',\qquad \dim\Row(B)=d'',\qquad \Row(A)=\Row(B)^\perp.
\end{equation}

\medskip

We begin with a basic criterion for when a finitely generated cone is in fact a linear subspace,
phrased in terms of a strictly positive linear dependence among its generators.

\begin{lemma}[Criterion for $\pos= \lin$]\label{lem:pos_eq_lin_hull}
Let $v_1,\dots,v_m\in\mathbb R^n$ be vectors. Then the following are equivalent:
\begin{enumerate}
  \item[(i)] $\pos(v_1,\dots,v_m)=\lin(v_1,\dots,v_m)$.
  \item[(ii)] There exist scalars $\lambda_1>0,\dots,\lambda_m>0$ such that
  $
    \lambda_1 v_1+\cdots+\lambda_m v_m = 0.
  $
\end{enumerate}
\end{lemma}

\begin{proof}
(i)$\Rightarrow$(ii). Assume~(i). Then for each $k\in[m]$ we have
$-v_k\in\lin(v_1,\dots,v_m)=\pos(v_1,\dots,v_m)$, so there exist coefficients
$\mu_{1;k}\geq 0,\dots,\mu_{m;k}\geq 0$ with
\[
  -v_k = \sum_{j=1}^m \mu_{j;k}\, v_j.
\]
Summing these identities over $k=1,\dots,m$ and rearranging yields $\sum_{j=1}^m \lambda_j v_j = 0$ with $\lambda_j := 1+\sum_{k=1}^m \mu_{j;k} \geq 1$. Hence $\lambda_1>0,\dots,\lambda_m>0$, proving (ii).

\vspace*{2mm}
\noindent
(ii)$\Rightarrow$(i). Assume $\sum_{j=1}^m \lambda_j v_j=0$ with all $\lambda_j>0$. Then for every $k\in[m]$ we have
\[
  -v_k = \sum_{j: j\neq k}(\lambda_j/\lambda_k) \, v_j \in \pos(v_1,\dots,v_m),
\]
implying $\pm v_k\in\pos(v_1,\dots,v_m)$ for every $k$. Therefore
$\lin(v_1,\dots,v_m)\subseteq \pos(v_1,\dots,v_m)$. The reverse inclusion
$\pos(v_1,\dots,v_m)\subseteq \lin(v_1,\dots,v_m)$ always holds, proving~(i).
\end{proof}

The next lemma makes precise how faces of the cone $\pos(a_1,\dots,a_n)$ correspond to
strictly positive dependent subsets (those for which $\pos = \lin$) in a Gale-dual configuration. It can be found in~\cite[Theorem~1 on p.~167]{Shephard1971DiagramsPositiveBases} or~\cite[2A11 on p.~95]{mcmullen_transforms}.

\begin{lemma}[Faces vs.\ $\pos=\lin$]\label{lem:faces-gale-dual}
Let $d',d''\in\mathbb N$ and put $n:=d'+d''$. Let $a_1,\dots,a_n\in\mathbb R^{d'}$ and
$b_1,\dots,b_n\in\mathbb R^{d''}$ be in linear Gale duality, and set $C_A:=\pos(a_1,\dots,a_n)$.
For $I\subseteq[n]$, the following are equivalent:
\begin{enumerate}
  \item[(i)] There exists a face $F$ of $C_A$ such that  $I = \{i\in [n]: a_i\in F\}$.
  \item[(ii)] $\pos(b_j:j\in I^c)=\lin(b_j:j\in I^c)$. Here $I^c := [n] \bsl I$ is the complement of $I$.
\end{enumerate}
\end{lemma}

\begin{proof}
Let $A\in\mathbb R^{d'\times n}$ and $B\in\mathbb R^{d''\times n}$ be the matrices with columns
$a_1,\dots,a_n$ and $b_1,\dots,b_n$, respectively.

\smallskip
\noindent
\emph{Step 1.}
Every face of $C_A$ is of the form $F=\{y\in C_A:\langle u,y\rangle=0\}$ for some $u\in\mathbb R^{d'}$ such that $\langle u,y\rangle\ge 0$ for all $y\in C_A$. (When $F= C_A$, we may take $u=0$.) Thus, (i) holds if and only if there exists $u\in\mathbb R^{d'}$ such that
\[
\langle u,a_i\rangle=0 \quad  (\text{for } i\in I)
\qquad\text{and}\qquad
\langle u,a_j\rangle>0 \quad (\text{for } j\in I^c).
\]
With $x:=u^\top A\in\mathbb R^n$,  (i) is equivalent to the existence of a vector
\begin{equation}\label{eq:x-in-rowA}
x\in\Row(A)\quad\text{ with }\quad  x_i=0 \quad  (\text{for } i\in I), \quad  x_j>0 \quad (\text{for }j\in I^c).
\end{equation}

\smallskip\noindent
\emph{Step 2.}
Apply Lemma~\ref{lem:pos_eq_lin_hull} to the vectors $b_j, j\in I^c$. Then (ii) is equivalent to the existence of
coefficients $x_j>0$ for $j\in I^c$ such that $\sum_{j\in I^c} x_j b_j=0$. Extending by $x_i:=0$ for $i\in I$ and defining $x= (x_1,\ldots, x_n)$, we see that (ii) is equivalent to the existence of a vector
\begin{equation}\label{eq:x-in-rowB}
x\in \Row(B)^\perp \quad\text{ with }\quad  x_i=0 \quad  (\text{for } i\in I), \quad  x_j>0 \quad (\text{for }j\in I^c).
\end{equation}

\smallskip\noindent
\emph{Step 3.}
By~\eqref{eq:Gale_duality_row_spaces_orthogonal}, $\Row(A) = \Row(B)^\perp$, proving the equivalence of~\eqref{eq:x-in-rowA} and~\eqref{eq:x-in-rowB}.   Hence (i) and (ii) are equivalent.
\end{proof}

Recall that a finite collection of vectors
$v_1,\dots,v_m\in\mathbb R^d$, where $m\ge d$,
is in general linear position if, for every
$J\subseteq[m]$ with $\#J=d$, the vectors
$(v_j)_{j\in J}$ are linearly independent.

\begin{lemma}\label{lem:gen_lin_pos_under_gale_duality}
Let $d',d''\in\mathbb N$ and put $n:=d'+d''$. Let $a_1,\dots,a_n\in\mathbb R^{d'}$ and
$b_1,\dots,b_n\in\mathbb R^{d''}$ be in linear Gale duality. If one of the configurations is in general linear position, then the other configuration is in general linear position as well.
\end{lemma}
\begin{proof}
Assume that $a_1,\dots,a_n$ are in general linear position.
Let $J\subseteq[n]$ with $\# J=d''$ and put $I:=J^c$, so $\# I=d'$.
To show that $b_j, j\in J$, are linearly independent, let $x\in\mathbb R^n$ satisfy $x_1b_1 + \ldots +x_n b_n =0$ and $x_i=0$ for all $i\in I$ (so $x$ is supported on $J$). Thus $x\in \Row(B)^\perp$ and, by~\eqref{eq:Gale_duality_row_spaces_orthogonal}, $x\in \Row(A)$. Let $A_I$ be the $(d'\times d')$-submatrix of $A$ consisting of the columns indexed by $I$. Since $\# I = d'$ and $a_1, \ldots, a_n$ are in general position, the matrix $A_I$ is invertible. Now, $x\in \Row(A)$ is a linear combination of the rows of $A$ and  $x_i=0$ for all $i\in I$. On the other hand, the rows of $A_I$ are linearly independent. It follows that $x$ is a \emph{trivial} linear combination of the rows of $A$ and hence $x=0$.  This shows that the columns $b_j, j\in J$, are linearly independent.
\end{proof}

Under the additional assumption of general linear position, the face-dependence correspondence
of Lemma~\ref{lem:faces-gale-dual} simplifies and can be phrased in terms of positive spanning
the whole space on the Gale-dual side. We are now ready to prove
Lemma~\ref{lem:faces-positive-spanning-intro-version}, which we restate for convenience.

\begin{lemma}[Faces vs.\ positive spanning subsets under general linear position]\label{lem:faces-positive-spanning}
Let $d',d''\in\mathbb N$ and put $n:=d'+d''$. Let $a_1,\dots,a_n\in\mathbb R^{d'}$ and
$b_1,\dots,b_n\in\mathbb R^{d''}$ be in linear Gale duality, and assume that both
configurations are in general linear position. Let $I\subseteq[n]$ satisfy $\#I\le d'-1$.
Then the following are equivalent:
\begin{enumerate}
  \item[(i)] $\pos(a_i:i\in I)$ is a face of $C_A:=\pos(a_1,\dots,a_n)$.
  \item[(ii)] $\pos(b_j:j\in I^c)=\mathbb R^{d''}$. Here $I^c := [n] \bsl I$ is the complement of $I$.
\end{enumerate}
\end{lemma}

\begin{proof}
We first consider the case $I=\varnothing$. Then $\pos(a_i:i\in I)=\{0\}$, so (i) says that
$\{0\}$ is a face of $C_A$.  Recall that $a_1,\ldots, a_n$ are in general linear position, in particular $a_i\neq 0$ for all $i\in[n]$.
Hence $\{0\}$ is a face of $C_A$ if and only if there exists a face $F$ of $C_A$ containing none
of the generators $a_1,\dots,a_n$.
Thus we may apply Lemma~\ref{lem:faces-gale-dual} with $I=\varnothing$,  which gives that (i) is equivalent to
\[
\pos(b_1,\ldots, b_n)=\lin(b_1,\ldots, b_n).
\]
By Definition~\ref{def:gale_duality_linear}, we have $\lin(b_1,\ldots, b_n)=\mathbb R^{d''}$.  Thus (i) is equivalent to  $\pos(b_1,\dots,b_n)=\mathbb R^{d''}$, which is exactly (ii). Hence the lemma holds for
$I=\varnothing$.

\smallskip\noindent
From now on  assume $1\le \# I\le d'-1$.

\smallskip\noindent
\emph{Claim.} If $a_1,\ldots, a_n$ are in general linear position and $I\subseteq [n]$ satisfies $1\le \# I\le d'-1$, then
\begin{equation}\label{eq:no-extra-generators-2}
\{k\in[n]: a_k\in \pos(a_i:i\in I)\}=I.
\end{equation}

\smallskip\noindent
\emph{Proof of the claim.} Clearly $I$ is contained in the set on the left. Conversely, let
$k\in[n]$ be such that $a_k\in\pos(a_i:i\in I)$. Then
$a_k\in\lin(a_i:i\in I)$.  If $k\notin I$, this gives a non-trivial linear dependence among the $\#I+1\le d'$ vectors
$a_i, i\in I \cup \{k\}$. Since $a_1,\ldots, a_n$ are in general linear position, every set of at
most $d'$ vectors is linearly independent, a contradiction. This proves~\eqref{eq:no-extra-generators-2}.
\qed

\smallskip\noindent
\emph{(i)$\Rightarrow$(ii).}
Assume $F:=\pos(a_i:i\in I)$ is a face of $C_A$. By the claim, $I = \{i\in [n]: a_i\in F\}$.  Hence Lemma~\ref{lem:faces-gale-dual} yields
\[
\pos(b_j:j\in I^c)=\lin(b_j:j\in I^c).
\]
Since $\#I\le d'-1$, we have $\# (I^c)=n-\#I\ge d''+1$. Since  $b_1,\ldots, b_n$ are in general linear
position, $\lin(b_j:j\in I^c)=\mathbb R^{d''}$. Therefore $\pos(b_j:j\in I^c)=\mathbb R^{d''}$, which is (ii).

\smallskip\noindent
\emph{(ii)$\Rightarrow$(i).}
Assume (ii). Then in particular $\pos(b_j:j\in I^c)= \lin(b_j:j\in I^c)$, so by
Lemma~\ref{lem:faces-gale-dual} there exists a face $F$ of $C_A$ such that $\{i\in [n]: a_i \in F\} = I$.
In particular, $\pos(a_i:i\in I)\subseteq F$.
For the reverse inclusion, let $x\in F\subseteq C_A$ and write $x=\sum_{k=1}^n \lambda_k a_k$ with $\lambda_k\ge 0$.
If $\lambda_k>0$ for some $k$, then $a_k\in F$ (because $F$ is a face of the cone $C_A$), hence
$k\in I$ recalling that $\{i\in [n]: a_i \in F\} = I$. Therefore $x$ is a nonnegative linear combination of
$a_i,i\in I$, i.e.\ $x\in \pos(a_i:i\in I)$. This shows $F\subseteq \pos(a_i:i\in I)$. Combining both inclusions yields $F=\pos(a_i:i\in I)$, so $\pos(a_i:i\in I)$ is a face of $C_A$, proving~(i).
\end{proof}

\section*{Acknowledgement}
This paper benefited from many helpful and illuminating discussions with ChatGPT, which sharpened both the perspective and presentation; the work would have been impossible without these exchanges.
Supported by the German Research Foundation under Germany's Excellence Strategy  EXC 2044/2 -- 390685587, \textit{Mathematics M\"unster: Dynamics - Geometry - Structure} and by the DFG priority program SPP 2265 \textit{Random Geometric Systems}.

\bibliography{rand_poly_cones_1_bib}
\bibliographystyle{plainnat}

\end{document}